\documentclass[a4paper,11pt]{article}

\usepackage{amsmath,amsthm,amsfonts,amssymb,graphics,epsfig,color}

\usepackage[left=2cm,top=3cm,right=2cm,bottom=3cm,bindingoffset=0.5cm]{geometry}
\usepackage[latin1]{inputenc}
\usepackage[english]{babel}
\usepackage{verbatim}
\usepackage{amscd,enumerate}
\usepackage{epstopdf}
\usepackage{graphicx}
\usepackage{multicol}

\newtheorem{thm}{Theorem}[section]

\newtheorem{lem}[thm]{Lemma}
\newtheorem{rem}[thm]{Remark}

\begin{document}

\title{Explicit Solutions for Distributed, Boundary and Distributed-Boundary Elliptic Optimal Control Problems}

\author{
Julieta Bollati \thanks{Departamento de Matem\'atica-CONICET, FCE,
Univ. Austral, Paraguay 1950, S2000FZF Rosario, Argentina. E-mail:
JBollati@austral.edu.ar; DTarzia@austral.edu.ar }\,\, Claudia M.
Gariboldi
\thanks{Departamento de Matem\'atica, FCEFQyN, Univ. Nac. de R\'io
Cuarto, Ruta 36 Km 601, 5800 R\'io Cuarto, Argentina. E-mail:
cgariboldi@exa.unrc.edu.ar} \,\, Domingo A. Tarzia $^\ast$
 }

\maketitle

\begin{abstract}

We consider a steady-state heat conduction problem in a
multidimensional bounded domain $\Omega$ for the Poisson equation with constant 
internal energy $g$ and mixed boundary conditions given by a constant
temperature $b$ in the portion $\Gamma_1$ of the boundary and a constant heat
flux $q$ in the remaining portion $\Gamma_2$  of the  boundary. Moreover,
we consider a family of steady-state heat conduction problems with a
convective condition on the boundary $\Gamma_1$ with heat transfer
coefficient $\alpha$ and external temperature $b$. We obtain
explicitly, for a rectangular domain in $\mathbb{R}^{2}$, an annulus in $\mathbb{R}^{2}$ and a spherical shell in $\mathbb{R}^{3}$,
the optimal controls, the system states and adjoint states for the
following optimal control problems: a \emph{distributed} control
problem on the internal energy $g$, a \emph{boundary} optimal
control problem on the heat flux $q$, a  \emph{boundary} optimal
control problem on the external temperature $b$ and a
\emph{distributed-boundary} simultaneous optimal control problem on
the source $g$ and the flux $q$. These explicit solutions can be
used for testing new numerical methods as a benchmark test. In agreement with theory, it is proved that the system state, adjoint state, optimal controls and optimal values corresponding to the problem with a convective condition on $\Gamma_1$ converge, when $\alpha\to\infty$, to the corresponding system state, adjoint state, optimal controls and optimal values that arise from the problem with a temperature condition on $\Gamma_1$. Also, we analyze the order of convergence in each case, which turns out to be $1/\alpha$ being new for these kind of elliptic optimal control problems.
\end{abstract}



{\thispagestyle{empty}} 

\section{Introduction}
The goal of this paper is to show the explicit solution for eight elliptic optimal control problems in two and three dimensional cases.

We consider a bounded domain $\Omega $ in ${\Bbb R}^{n}$ $(n=2,3)$, whose
regular boundary $\Gamma $ consist of the union of three disjoint portions $%
\Gamma _{1}$, $\Gamma _{2}$ and $\Gamma_{3}$ with $meas(\Gamma_{1})>0$,
$meas(\Gamma_{2})>0$ and $meas(\Gamma_{3})\geq 0$. We present the following steady-state heat
conduction problems $S$ and $S_{\alpha }$ (for each parameter
$\alpha >0)$ respectively, with mixed boundary conditions
\begin{equation}
\Delta u=g,\quad \text{in }\Omega \qquad \qquad u\big|_{\Gamma
_{1}}=b,\qquad \qquad-\frac{\partial u}{\partial n}\big|_{\Gamma
_{2}}=q,\qquad \qquad \frac{\partial u}{\partial n}\big|_{\Gamma
_{3}}=0,\label{P}
\end{equation}
\begin{equation}
-\Delta u_\alpha=g\,\ \quad\text{in }\Omega \quad \qquad-\frac{\partial u_\alpha}{\partial n%
}\big|_{\Gamma _{1}}=\alpha (u-b),\quad \qquad-\frac{\partial
u_\alpha}{\partial n}\big|_{\Gamma _{2}}=q,\quad \qquad  \frac{\partial u_{\alpha}}{\partial n}\big|_{\Gamma
_{3}}=0, \label{Palfa}
\end{equation}
where $g$ is the internal energy in $\Omega$, $b$ is the temperature
on $ \Gamma_{1}$ for (\ref{P}) and the temperature of the external
neighborhood of $\Gamma_{1}$ for (\ref{Palfa}), $q$ is the heat flux
on $\Gamma_{2}$ and $\alpha >0$ is the heat transfer coefficient on
$\Gamma_{1}$. The above problems can be
considered as the steady-state Stefan problems, \cite{Gar91,TaTa,Ta1,Ta88}.  Note that mixed boundary conditions play an
important role in various applications, e.g. heat conduction and electric potential problems \cite{HaMe}. In general, the solution of a mixed elliptic boundary problems is not so regular \cite{Gris85} but there exist some examples which solutions are regular \cite{Az82,La08,Sh68}.

Let $u$ and $u_{\alpha}$ the unique solutions of the elliptic
problems (\ref{P}) and (\ref{Palfa}), respectively.
In relation with these state systems, we present the particular eight following
optimal control problems \cite{Bar84,Li,Ne06,Tro}.

\subsection{Distributed optimal control on the constant internal energy $g$}

Following \cite{GT1}, we consider the distributed optimal control
problems:
\begingroup
\addtolength{\jot}{0.5em}
\begin{align}
& \text{ find }\quad g_{op}\in \mathbb{R}\quad\text{ such that }\quad
J_{1}(g_{op})=\min\limits_{g\in \mathbb{R}} \text{ }J_{1}(g)
\label{PControlJ1} \\
& \text{find }\quad g_{{\alpha}_{op}}\in \mathbb{R}\quad\text{ such that
}\quad J_{1\alpha}(g_{{\alpha}_{op}})=\min\limits_{g\in \mathbb{R}} \text{
}J_{1\alpha}(g)  \label{PControlalfaJ1}
\end{align}
\endgroup
with $J_{1}:\mathbb{R}{\rightarrow}{\Bbb R}_{0}^{+}$ and
$J_{1\alpha}:\mathbb{R}{\rightarrow}{\Bbb R}_{0}^{+}$, given by
\[
J_{1}(g)=\frac{1}{2}\left\|u_{g}-z_{d}\right\|
_{H}^{2}+\frac{M_{1}}{2}\left\| g\right\|_{H}^{2}\quad
\text{and}\quad J_{1\alpha}(g)=\frac{1}{2}\left\|u_{\alpha
g}-z_{d}\right\| _{H}^{2}+\frac{M_{1}}{2}\left\| g\right\|_{H}^{2}
\]
with $H=L^2(\Omega)$, and where $u_{g}$ and $u_{\alpha g}$ denote the unique solutions of the
problems (\ref{P}) and (\ref{Palfa})
respectively, for data $q\in \mathbb{R}$, $b\in \mathbb{R}$, $z_{d}\in \mathbb{R}$
and $M_{1}$ a positive constant.

\subsection{Boundary optimal control on the constant heat flux $q$ on $\Gamma_{2}$}

Following \cite{GT2}, we formulate the boundary optimal control
problems:
\begingroup
\addtolength{\jot}{0.5em}
\begin{align}
& \text{find }\quad q_{op}\in \mathbb{R}\quad\text{ such that }\quad
J_{2}(q_{op})=\min\limits_{q\in \mathbb{R}}\,J_{2}(q) \label{PControlJ2} \\
& \text{find }\quad q_{{\alpha}_{op}}\in \mathbb{R}\quad\text{ such that
}\quad J_{2\alpha}(q_{{\alpha}_{op}})=\min\limits_{q\in\mathbb{R}}\,J_{2\alpha}(q)  \label{PControlalfaJ2}
\end{align}
\endgroup
where 
$J_{2}:\mathbb{R}{\rightarrow}{\Bbb R}_{0}^{+}$ and
$J_{2\alpha}:\mathbb{R}{\rightarrow}{\Bbb R}_{0}^{+}$ given by
\[
J_{2}(q)=\frac{1}{2}\left\| u_{q}-z_{d}\right\|
_{H}^{2}+\frac{M_{2}}{2}\left\| q\right\|_{Q}^{2} \quad
\text{and}\quad J_{2\alpha}(q)=\frac{1}{2}\left\| u_{\alpha
q}-z_{d}\right\| _{H}^{2}+\frac{M_{2}}{2}\left\| q\right\|_{Q}^{2}
\]
with $Q=L^2\left(\Gamma_2 \right)$ where $u_{q}$ y $u_{\alpha q}$ are the unique solutions of the
problems (\ref{P}) and (\ref{Palfa})
respectively,  for data $g\in \mathbb{R}$, $b\in \mathbb{R}$, $z_d\in
\mathbb{R}$ and $M_2$ a positive constant.

\subsection{Boundary optimal control on the constant temperature $b$ in an external neighborhood of $\Gamma_{1}$}

Following \cite{BEM}, we consider the boundary optimal control
problems:
\begingroup
\addtolength{\jot}{0.5em}
\begin{align}
& \text{find }\quad b_{op}\in \mathbb{R} \quad\text{ such
that }\quad J_{3}(b_{op})=\min\limits_{b\in
\mathbb{R}}\,J_{3}(b)
\label{PControlJ3} \\
& \text{find }\quad b_{{\alpha}_{op}}\in
\mathbb{R}\quad\text{ such that }\quad
J_{3\alpha}(b_{{\alpha}_{op}})=\min\limits_{b\in
\mathbb{R}}\,J_{3\alpha}(q)  \label{PControlalfaJ3}
\end{align}
\endgroup
with $J_{3}:\mathbb{R}{\rightarrow}{\Bbb R}_{0}^{+}$ and
$J_{3\alpha}:\mathbb{R}{\rightarrow}{\Bbb R}_{0}^{+}$, given by
\[
J_{3}(b)=\frac{1}{2}\left\| u_{b}-z_{d}\right\|
_{H}^{2}+\frac{M_{3}}{2}\left\| b\right\|_{B}^{2}
\]
\[J_{3\alpha}(b)=\frac{1}{2}\left\| u_{\alpha b}-z_{d}\right\|
_{H}^{2}+\frac{M_{3}}{2}\left\| b\right\|_{B}^{2}
\]
with $B=L^2\left(\Gamma_1 \right)$, where $u_{b}$ y $u_{\alpha b}$ are the unique solutions of the
problems (\ref{P}) and (\ref{Palfa})
respectively, for data $g\in \mathbb{R}$, $q\in\mathbb{R}$, $z_d\in \mathbb{R}$
and $M_3$ a positive constant.

\subsection{Simultaneous distributed-boundary optimal control on the constant source $g$ and the constant flux $q$}

Following \cite{GT3}, we formulate the simultaneous distributed-boundary
optimal control problems:
\begingroup
\addtolength{\jot}{0.5em}
\begin{align}
& \text{find }\,\,(g,q)_{op}\in \mathbb{R}\times \mathbb{R}\,\,\text{ such that
}\,\, J_{4}((g,q)_{op})=\min\limits_{g\in \mathbb{R},q\in \mathbb{R}} \text{
}J_{4}(g,q)  \label{PControl} \\
& \text{find }\, (g,q)_{{\alpha}_{op}}\in \mathbb{R}\times \mathbb{R}\text{ such
that}\, J_{4\alpha }((g,q)_{{\alpha}_{op}})=\min\limits_{g\in \mathbb{R},q\in
\mathbb{R}}J_{4\alpha }(g,q)  \label{PControlalfa}
\end{align}
\endgroup
with the cost functional $J_{4}:\mathbb{R}\times \mathbb{R}{\rightarrow}{\Bbb
R}_{0}^{+}$ and $J_{4\alpha}:\mathbb{R}\times \mathbb{R}{\rightarrow}{\Bbb R}_{0}^{+}$
given by
\begingroup
\addtolength{\jot}{0.5em}
\begin{align}
& J_{4}(g,q)=\frac{1}{2}\left\| u_{(g,q)}-z_{d}\right\|
_{H}^{2}+\frac{M_{4}}{2}\left\|
g\right\|_{H}^{2}+\frac{M_{5}}{2}\left\| q\right\|_{Q}^{2}\nonumber \\
& J_{4\alpha}(g,q)=\frac{1}{2}\left\| u_{\alpha (g,q)}-z_{d}\right\|
_{H}^{2}+\frac{M_{ 4}}{2}\left\|
g\right\|_{H}^{2}+\frac{M_{5}}{2}\left\| q\right\| _{Q}^{2}\nonumber
\end{align}
\endgroup
where $u_{(g,q)}$ and $u_{\alpha (g,q)}$ are the unique solutions of the problems (\ref{P}) and (\ref{Palfa})
respectively, for data $b \in \mathbb{R}$, $z_d\in \mathbb{R}$,
$M_{4}$ and $M_{5}$ positive constants.

\subsection{Adjoint states}

We define the adjoint state corresponding to problems $S$ and $S_{\alpha}$  as the unique solution of the following mixed elliptic problems, respectively.
\begin{equation}
-\Delta p=u-z_d,\quad \text{in }\Omega \quad \qquad p\big|_{\Gamma
_{1}}=0,\qquad\quad\frac{\partial p}{\partial n}\big|_{\Gamma
_{2}}=0,\quad \qquad \frac{\partial p}{\partial n}\big|_{\Gamma_3}=0, \label{Adjointp}
\end{equation}
and
\begin{equation}
-\Delta p_{\alpha}=u_{\alpha}-z_d,\quad \text{in }\Omega\qquad\quad-\frac{\partial p_\alpha}{\partial n%
}\big|_{\Gamma _{1}}=\alpha p_{\alpha},\quad \qquad \frac{\partial
p_\alpha}{\partial n}\big|_{\Gamma _{2}}=0,\qquad \qquad \frac{\partial p_\alpha}{\partial n}\big|_{\Gamma_3}=0\label{Adjointpalpha}
\end{equation}
with $u$ and $u_\alpha$ given by the unique solution of (\ref{P}) and (\ref{Palfa}), respectively.
Other theoretical optimal control problems in the subject was done in
\cite{BEN,Be97,Ca08,Ca09,CaRay, Ca02,   HiHi,Hi, Ke99,Me06,Wa11}.

\par In \cite{BEM,GT1,GT2,GT3} were obtained results of existence and uniqueness of the optimal controls, as well also convergence results, when the heat transfer coefficient $\alpha$ goes to infinity, of the optimal controls, the system states and the adjoint states, in suitable Sobolev spaces.

\par In Section 2, we calculate explicitly the optimal controls, the system states
and the adjoint states, for the optimal control problems  previously
formulated, related to $S$ and $S_{\alpha}$ respectively, in a rectangular domain in $\mathbb{R}^{2}$. In Section
3 and Section 4, similar results are obtained in an annulus in
$\mathbb{R}^{2}$ and a spherical shell in $\mathbb{R}^{3}$,
respectively. In all cases, we  obtain, in agreement with theory, the convergence of the optimal controls and values when $\alpha\to \infty$ as it was obtained in \cite{BEM, GT1,GT2,GT3} and for numerical analysis in \cite{TaIfip}. Also, the corresponding rates of convergence are   studied, obtaining, in Appendix $A$, that the order of convergence in each case is $1/\alpha$ which is new for these elliptic optimal control problems.

We remark that the expressions for the system states $u$, $u_{\alpha}$, the adjoint states $p$, $p_{\alpha}$, the functional cost $J_i, J_{i \alpha}$, $i=1,. .,4$, and the optimal controls are defined for each particular domain, using the same notation.

\section{Optimal solutions for a rectangle in $\mathbb{R}^{2}$ }

In this Section, we consider a rectangular domain in the plane, that
is
\[
\Omega_{1}=\{(x,y)\in \mathbb{R}^{2}: 0 < x< x_{0},\,\,0 < y <
y_{0}\}
\]
whose boundaries $\Gamma_{1}$, $\Gamma_{2}$ and $\Gamma_{3}$ are
given by (see Figure 1):
\[
\Gamma_{1}=\{(x,y)\in\mathbb{R}^{2}: x=0, \,\, 0\leq y\leq y_{0}\}
\]
\[
\Gamma_{2}=\{(x,y)\in\mathbb{R}^{2}: x=x_{0}, \,\, 0\leq y\leq
y_{0}\}
\]
\[
\Gamma_{3}=\{(x,y)\in\mathbb{R}^{2}: y=0,\, 0< x< x_{0}\}\cup\,\,
\{(x,y)\in\mathbb{R}^{2}: y=y_{0},\, 0< x< x_{0}\}
\]
 \begin{center}
  \quad  \includegraphics[scale=0.5]{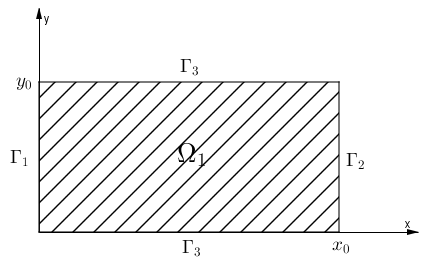}\\
    {\small  Figure 1}
 \end{center}

If we consider constant data $g$, $b$, $\alpha$, $q$ and the desired
system state $z_d\in \mathbb{R}$, we obtain the following result,
which proof is omitted:
\begin{lem}
i) The system state and adjoint state for the problem (\ref{P}) and (\ref{Adjointp}) respectively are
given by:
\[
u(x,y)=u(x)=-g\frac{x^2}{2}+(gx_0-q)x+b
\]
\[
p(x,y)=p(x)=g\dfrac{x^4}{24}-(gx_0-q)\dfrac{x^3}{6}-(b-z_d)\dfrac{x^2}{2}+Ax
\]
where
$A=x_0 \Big[g\frac{x_0^2}{3}-q\frac{x_0}{2}+(b-z_d) \Big]$.\\
ii) The system state and adjoint state for the problem (\ref{Palfa}) and (\ref{Adjointpalpha}) respectively
take the expressions:
\[
u_{\alpha }(x,y)=u_{\alpha
}(x)=-g\frac{x^2}{2}+(gx_0-q)x+\frac{gx_0-q}{\alpha}+ b
\]
\[
p_{\alpha }(x,y)=p_{\alpha
}(x)=g\dfrac{x^4}{24}-(gx_0-q)\dfrac{x^3}{6}-\left(\frac{gx_0-q}{\alpha}+(b-z_d)\right)\frac{x^2}{2}+
A_{\alpha }x+\frac{A_{\alpha }}{\alpha}
\]
where
$A_{\alpha}
=x_0\left[ gx_0^2 \left( \frac{1}{3}+\frac{1}{\alpha x_0}\right)-q x_0 \left( \frac{1}{2}-\frac{1}{\alpha x_0}\right)+(b-z_d)\right]$.
\end{lem}

\begin{rem}\label{convergenciaA}
It is immediate that  $u_{\alpha}$ converges to $u$ and $p_{\alpha}$ to $p$, when $\alpha\to\infty$. 
Moreover, we can prove that there exists a positive constant $K_1=K_1(x_0,y_0,g,q)$ such that:

$$\vert\vert u_{\alpha}-u\vert\vert_{H^1(\Omega_1)}=\vert\vert u_{\alpha}-u\vert\vert_{L^2(\Omega_1)}= \dfrac{K_1}{\alpha}$$
where
$$K_1=(x_0 y_0)^{1/2} \vert q-gx_0\vert.$$
In the same way, a similar estimate can be obtained for the adjoint states $p_{\alpha}$ and $p$. It can be proved that there exists a positive constant $L_1=L_1(x_0,y_0,g,q,b,z_d)$
such that:
$$\lim\limits_{\alpha\to\infty} \alpha \vert\vert p_{\alpha}-p\vert\vert_{L^2(\Omega_1)}=L_1$$
where
\begin{align*}
L_1& =\left\lbrace\tfrac{x_0^3  y_0}{180}  \Big|180 (b+z_d)^2 + 129 q^2 x_0^2 - 208 g q x_0^3 + 84 g^2 x_0^4 - 
  \right.\\
 &\left.  - 60 (b-z_d) (5 q x_0 - 4 g x_0^2 ) \Big|\right\rbrace^{1/2}.
\end{align*}

\end{rem}

Next, we present the following lemma that will allow us to find the solution of the optimal control problems:

\begin{lem}\label{normarectangulo2}

i) For the problem (\ref{P}), it can be obtained that:
\begin{align*}
\frac{1}{2}  \Vert u-z_d\Vert^2_{L^2(\Omega_1)}&=\,\frac{y_0}{2} \Big[ C_1  g^2 x_0^5+ C_2  q^2 x_0^3+C_3 x_0  (b-z_d)^2 +C_4  gq x_0^4  \\
 &+C_5  g x_0^3 (b-z_d)+C_6  q x_0^2 (b-z_d)\Big]
\end{align*}
with:
$$
C_1=\frac{2}{15} , \quad C_2=\frac{1}{3}, \quad C_3=1, \quad C_4=-\frac{5}{12},\quad C_5=\frac{2}{3}, \quad C_6=-1.
$$

\noindent ii) For the problem (\ref{Palfa}), we have:
\begin{align*}
\frac{1}{2}  \Vert u_\alpha-z_d\Vert^2_{L^2(\Omega_1)}&= \frac{y_0}{2} \Big[ C_{1\alpha}  g^2 x_0^5+ C_{2\alpha}  q^2 x_0^3+C_{3\alpha} x_0 (b-z_d)^2  +C_{4\alpha}  gq x_0^4\\
&+C_{5\alpha}  g x_0^3 (b-z_d)+C_{6\alpha}  q x_0^2 (b-z_d)\Big]
\end{align*}
with:

\medskip
\begin{tabular}{lll}
$C_{1\alpha}=\frac{2}{15} +\frac{2}{3\alpha x_0}+\frac{1}{\alpha^2 x_0^2}$\qquad\quad\qquad  & $C_{2\alpha}= \frac{1}{3}+\frac{1}{\alpha x_0}+\frac{1}{\alpha^2 x_0^2} $\qquad\quad\qquad  &  $C_{3\alpha}=1=C_3 $ \\
 $ C_{4\alpha}= -\frac{5}{12}-\frac{5}{3\alpha x_0}-\frac{2}{\alpha^2 x_0^2}$ & $C_{5\alpha}=\frac{2}{3}+\frac{2}{\alpha x_0} $
& $C_{6\alpha}=- 1-\frac{2}{\alpha x_0}.$
\end{tabular}

\begin{rem}\label{convergenciaCi}
It is clear that $C_{i\alpha}$ converges to $C_i$, when \mbox{$\alpha\to\infty$} for $i=1,2,\dots,6$.
\end{rem}

\end{lem}

\begin{thm}\label{TeoremaRectangulo}
i) For the distributed optimal control problems (\ref{PControlJ1})
and (\ref{PControlalfaJ1}), the optimal controls are given by:
\begin{equation}\label{goptim}
g_{op}=-\frac{C_4 q x_0+ C_5 (b-z_d)}{2 x_0^2 \left(C_1+\tfrac{M_1}{x_0^4}\right)}
\end{equation}


\begin{equation}\label{galfaoptim}
g_{\alpha_{op}}=-\frac{C_{4\alpha} q x_0+C_{5\alpha} (b-z_d)}{2 x_0^2 \left(C_{1\alpha}+\tfrac{M_1}{x_0^4}\right)}
\end{equation}
and the optimal values are given by:
\begin{align}
J_1\left(g_{op} \right)&=\frac{x_0 y_0 }{8 \left(C_1+\frac{M_1}{x_0^4}\right)} \Big[4 \left(C_1+\tfrac{M_1}{x_0^4}\right) \Big(C_2 q^2 x_0^2+ C_3 (b-z_d)^2+C_6 q x_0 (b-z_d)\Big)\nonumber\\
&-\Big(C_4 q x_0+C_5 (b-z_d)\Big)^2\Big] \label{J1goptim}
\end{align}
and
\begin{align}
J_{1\alpha}\left(g_{\alpha_{op}} \right)&=\frac{x_0 y_0 }{8 \left(C_{1\alpha}+\frac{M_1}{x_0^4}\right)}\Big[4 \left(C_{1\alpha}+\tfrac{M_1}{x_0^4}\right) \Big(C_{2\alpha} q^2 x_0^2+ C_{3\alpha} (b-z_d)^2+C_{6\alpha} q x_0 (b-z_d)\Big) \nonumber\\
&-\Big(C_{4\alpha} q x_0+ C_{5\alpha} (b-z_d)\Big)^2\Big] \label{J1galfaoptim}
\end{align}


\noindent ii) For the boundary optimal control problems (\ref{PControlJ2}) and
(\ref{PControlalfaJ2}), the optimal controls are given by:
\begin{equation}\label{qoptim}
q_{op}=-\frac{C_4 g x_0^2+C_6 (b-z_d)}{2 x_0 \left(C_2+\tfrac{M_2}{x_0^3}\right)}
\end{equation}

\begin{equation}\label{qalfaoptim}
q_{\alpha_{op}}=-\frac{C_{4\alpha} g x_0^2+ C_{6\alpha} (b-z_d)}{2 x_0 \left(C_{2\alpha}+\frac{M_2}{x_0^3}\right)}
\end{equation}
and the optimal values can be expressed as:
\begin{align}
J_2\left(q_{op}\right)&=\frac{x_0 y_0 }{8 \left(C_2+\frac{M_2}{x_0^3}\right)}\Big[4 \Big(C_2+\frac{M_2}{x_0^3}\Big) \Big(C_1 g^2 x_0^4+C_3 (b-z_d)^2+C_5 g x_0^2 (b-z_d)\Big)\nonumber\\
&-\Big(C_4 g x_0^2+C_6 (b-z_d)\Big)^2\Big] \label{J2qoptim}
\end{align}
and
\begin{align}
J_{2\alpha}\left(q_{\alpha_{op}}\right)&=\frac{x_0 y_0 }{8 \left(C_{2\alpha}+\frac{M_2}{x_0^3}\right)} \Big[4 \Big(C_{2\alpha}+\frac{M_2}{x_0^3}\Big) \Big(C_{1\alpha} g^2 x_0^4+C_{3\alpha} (b-z_d)^2+C_{5\alpha} g x_0^2 (b-z_d)\Big) \nonumber \\
&-\Big(C_{4\alpha} g x_0^2+C_{6\alpha} (b-z_d)\Big)^2\Big]
\label{J2qalfaoptim}
\end{align}

\noindent iii) For the boundary optimal control problems (\ref{PControlJ3})
and (\ref{PControlalfaJ3}), the optimal controls are given by:
\begin{equation}\label{boptim}
b_{op}=-\frac{C_5 g x_0^2+C_6 q x_0-2 C_3 z_d}{2 \left(C_3+\frac{M_3}{x_0}\right)}
\end{equation}
\begin{equation}\label{balfaoptim}
b_{\alpha_{op}}=-\frac{C_{5\alpha} g x_0^2+C_{6\alpha} q x_0-2 C_{3\alpha} z_d}{2 \left(C_{3\alpha}+\frac{M_3}{x_0}\right)}
\end{equation}
and the optimal values are:
\begin{align}
J_3\left(b_{op}\right)&= \frac{x_0 y_0 }{8 \left(C_3+\tfrac{M_3}{x_0}\right)}\Big[4 \Big(C_3+\frac{M_3}{x_0}\Big) \Big(C_1 g^2 x_0^4+C_2 q^2 x_0^2+C_3 z_d^2+C_4 g q x_0^3-C_5 g x_0^2 z_d-C_6 q x_0 z_d\Big) \nonumber\\
&-\Big(-2 C_3 z_d+C_5 g x_0^2+C_6 q x_0\Big)^2 \Big]
\label{J3boptim}
\end{align}
and
\begin{align}
J_{3\alpha}\left(b_{\alpha_{op}}\right)&= \frac{x_0 y_0}{8 \left(C_{3\alpha}+\tfrac{M_3}{x_0}\right)}  \Big[ 4 \Big(C_{3\alpha}+\frac{M_3}{x_0}\Big) \Big(C_{1\alpha} g^2 x_0^4+C_{2\alpha} q^2 x_0^2+C_{3\alpha} z_d^2+C_{4\alpha} g q x_0^3-C_{5\alpha} g x_0^2 z_d\nonumber \\
& -C_{6\alpha} q x_0 z_d\Big)-\Big(-2 C_{3\alpha} z_d+C_{5\alpha} g x_0^2+C_{6\alpha} q x_0\Big)^2\Big].\label{J3balfaoptim}
\end{align}

\noindent iv) For the distributed-boundary optimal control problem
(\ref{PControl}) and (\ref{PControlalfa}), the optimal solutions are
given by:
\begin{equation}
(g,q)_{op}=(g^{op},q^{op})=\left(
\frac{(b-z_d)}{x_0^2}\Delta_1, \frac{(b-z_d)}{x_0} \Pi_1\right)\label{gqoptim}
\end{equation}
where
 $$\Delta_1=\frac{C_4 C_6-2 C_5 \left(C_2+\frac{M_5}{x_0^3}\right)}{4 \left(C_1+\tfrac{M_4}{x_0^4}\right) \left(C_2+\tfrac{M_5}{x_0^3}\right)-C_4^2}, \qquad \Pi_1=\frac{C_4 C_5-2 C_6 \left(C_1+\tfrac{M_4}{x_0^4}\right)}{4 \left(C_1+\tfrac{M_4}{x_0^4}\right) \left(C_2+\tfrac{M_5}{x_0^3}\right)-C_4^2}$$
and
\begin{equation}
(g,q)_{\alpha_{op}}=(g_{\alpha}^{op},q_{\alpha}^{op})=\left(\frac{(b-z_d)}{x_0^2} \Delta_{1_{\alpha}},\frac{(b-z_d)}{x_0} \Pi_{1_{\alpha}}\right)\label{gqalfaoptim}
\end{equation}
with
$$\Delta_{1_{\alpha}}=\frac{C_{4\alpha} C_{6\alpha}-2 C_{5\alpha} \left(C_{2\alpha}+\frac{M_5}{x_0^3}\right)}{ 4 \left(C_{1\alpha}+\tfrac{M_4}{x_0^4}\right) \left(C_{2\alpha}+\tfrac{M_5}{x_0^3}\right)-C_{4\alpha}^2},\qquad \Pi_{1_{\alpha}}= \frac{C_{4\alpha} C_{5\alpha}-2 C_{6\alpha} \left(C_{1\alpha}+\tfrac{M_4}{x_0^4}\right)}{ 4 \left(C_{1\alpha}+\tfrac{M_4}{x_0^4}\right) \left(C_{2\alpha}+\tfrac{M_5}{x_0^3}\right)-C_{4\alpha}^2}$$
obtaining the following optimal values:
\begin{align}
J_4\left(g^{op},q^{op}\right)&=  \frac{x_0 y_0(b-z_d)^2 }{2 \Big(C_4^2-4 \left(C_1+\tfrac{M_4}{x_0^4}\right) \left(C_2+\tfrac{M_5}{x_0^3}\right)\Big)}\Big[-4 C_3 \left(C_1+\tfrac{M_4}{x_0^4}\right) \left(C_2+\tfrac{M_5}{x_0^3}\right)\nonumber \\
&+C_6^2 \left(C_1+\tfrac{M_4}{x_0^4}\right)+C_5^2 \left(C_2+\tfrac{M_5}{x_0^3}\right)+C_3 C_4^2-C_4 C_5 C_6\Big]
 \label{J4gqoptim}
\end{align}
and
\begin{align}
J_{4\alpha}\left(g_{\alpha}^{op},q_{\alpha}^{op}\right)&= \frac{ x_0 y_0(b-z_d)^2}{2 \left(C_{4\alpha}^2-4 \left(C_{1\alpha}+\tfrac{M_4}{x_0^4}\right) \left(C_{2\alpha}+\tfrac{M_5}{x_0^3}\right)\right)} \Big[   -4 C_{3\alpha} \left(C_{1\alpha}+\tfrac{M_4}{x_0^4}\right) \left(C_{2\alpha}+\tfrac{M_5}{x_0^3}\right)\nonumber \\
&+ C_{6\alpha}^2 \left(C_{1\alpha}+\tfrac{M_4}{x_0^4}\right)+C_{5\alpha}^2 \left(C_{2\alpha}+\tfrac{M_5}{x_0^3}\right)+C_{3\alpha} C_{4\alpha}^2-C_{4\alpha} C_{5\alpha} C_{6\alpha}\Big]\label{J4afagqalfaoptim}
\end{align}

\noindent v) When $\alpha\to\infty$ the following convergences and estimates hold:

\begin{enumerate}[a)]
\item $g_{\alpha_{op}}\to g_{op}$ \quad  with\quad $\vert g_{\alpha_{op}}- g_{op}\vert=\mathcal{O}\left(\frac{1}{\alpha}\right)$
\item  $q_{\alpha_{op}}\to q_{op}$ \quad  with\quad $\vert q_{\alpha_{op}}- q_{op}\vert=\mathcal{O}\left(\frac{1}{\alpha}\right) $ 
\item  $b_{\alpha_{op}}\to b_{op}$ \quad  with\quad  $\vert b_{\alpha_{op}}- b_{op}\vert=\mathcal{O}\left(\frac{1}{\alpha}\right)$ 
\item  $(g,q)_{\alpha_{op}}\to (g,q)_{op}$ \quad  with\quad $\vert g_{\alpha}^{op}- g^{op}\vert=\mathcal{O}\left(\frac{1}{\alpha}\right) \quad \text{and}\quad \vert q_{\alpha}^{op}- q^{op}\vert=\mathcal{O}\left(\frac{1}{\alpha}\right)$

\end{enumerate}

Moreover, when $\alpha\to\infty$,  we have:
\begin{enumerate}[a')]

\item $J_{1\alpha}\left(g_{\alpha_{op}} \right)\to J_1(g_{op})$ \quad with \quad $\vert J_{1\alpha}\left(g_{\alpha_{op}} \right)- J_1(g_{op})\vert=\mathcal{O}\left(\frac{1}{\alpha}\right)$
\item $J_{2\alpha}\left(q_{\alpha_{op}} \right)\to J_2(q_{op})$\quad with \quad $\vert J_{2\alpha}\left(q_{\alpha_{op}} \right)- J_2(q_{op})\vert=\mathcal{O}\left(\frac{1}{\alpha}\right)$
\item $J_{3\alpha}\left(b_{\alpha_{op}} \right)\to J_3(b_{op})$\quad with \quad $\vert J_{3\alpha}\left(b_{\alpha_{op}} \right)- J_3(b_{op})\vert=\mathcal{O}\left(\frac{1}{\alpha}\right)$
\item $J_{4\alpha}\left((g,q)_{\alpha_{op}} \right)\to J_4((g,q)_{op})$\quad with \quad $\vert J_{4\alpha}\left((g,q)_{\alpha_{op}} \right)- J_4((g,q)_{op})\vert=\mathcal{O}\left(\frac{1}{\alpha}\right)$.
\end{enumerate}

\end{thm}
\begin{proof}
\noindent i) Taking into account that the functional $J_{1}$ and
$J_{1\alpha}$ are given by the following quadratic forms
\begin{align*}
J_1(g)& =  \dfrac{y_0}{2}\Big[ g^2 \Big(C_1 x_0^5+M_1 x_0 \Big)+g \Big(C_4 q  x_0^4+C_5 x_0^3 (b-z_d) \Big)\nonumber\\
 &+   C_2  q^2 x_0^3 +C_3  x_0 (b-z_d)^2 +C_6 q x_0^2 (b-z_d)\Big]\nonumber
\end{align*}
and
\begin{align}
J_{1\alpha}(g)& =  \dfrac{y_0}{2}\Big[ g^2 \Big(C_{1\alpha} x_0^5+M_1 x_0 \Big)+g \Big(C_{4\alpha} q  x_0^4+C_{5\alpha}  x_0^3 (b-z_d) \Big)\nonumber\\
 &+  C_{2\alpha}  q^2 x_0^3 +C_{3\alpha}   x_0 (b-z_d)^2 +C_{6\alpha} q x_0^2 (b-z_d)\Big]\nonumber
\end{align}
we obtain that the optimal solutions $g_{op}$ and $g_{\alpha_{op}}$
for the problems (\ref{PControlJ1}) and (\ref{PControlalfaJ1}) are
given by (\ref{goptim}) and (\ref{galfaoptim}), respectively since the second derivative is positive in both cases.

In addition, if we evaluate the functional $J_1$ at $g_{op}$  it is obtained formula (\ref{J1goptim}). In a similar way, computing $J_{1\alpha}$  at $g_{\alpha_{op}}$ it can be derived the closed form (\ref{J1galfaoptim}).

\noindent ii) The functional $J_2$ and $J_{2\alpha}$ are given by the
expressions:
\begin{align}
J_2(q)&=\dfrac{y_0}{2}\left[ q^2 \left(C_2 x_0^3+M_2 \right)+q \left(C_4 g x_0^4+C_6 x_0^2 (b-z_d) \right)\right.\nonumber \\
&+\left. C_1 g^2 x_0^5 +C_3 x_0 (b-z_d)^2+C_5 g x_0^3 (b-z_d) \right] \nonumber
\end{align}
and
\begin{align}
J_{2\alpha}(q)&=\dfrac{y_0}{2}\left[ q^2 \left(C_{2\alpha} x_0^3+M_2 \right)+q \left(C_{4\alpha} g x_0^4+C_{6\alpha} x_0^2 (b-z_d) \right)\right.\nonumber \\
&+\left. C_{1\alpha} g^2 x_0^5 +C_{3\alpha} x_0 (b-z_d)^2+C_{5\alpha} g x_0^3 (b-z_d) \right] \nonumber
\end{align}
and then the corresponding minimum are given by (\ref{qoptim}) and
(\ref{qalfaoptim}), respectively, since the second derivative is positive in both cases.  Evaluating $J_2$ and $J_{2\alpha}$ at $q_{op}$ and $q_{\alpha_{op}}$ respectively, and through computations, the formulas (\ref{J2qoptim}) and (\ref{J2qalfaoptim}) can be obtained.

\medskip

\noindent iii) For the problems (\ref{PControlJ3}) and (\ref{PControlalfaJ3}),
the functional $J_{3}$ and $J_{3\alpha}$ can be expressed as
\begin{align}
J_3(b)&=\dfrac{y_0}{2}\left[b^2 \left(C_3 x_0+M_3 \right)+b\left(-2 C_3 x_0z_d+C_5 g x_0^3+C_6 q x_0^2 \right) \right. \nonumber \\
&+\left. C_1 g^2 x_0^5 +C_2 q^2 x_0^3 +C_3 x_0 z_d^2 +C_4 g q x_0^4-C_5 g x_0^3 z_d-C_6 q x_0^2 z_d \right]\nonumber
\end{align}
and
\begin{align}
J_{3\alpha}(b)&=\dfrac{y_0}{2}\left[b^2 \left(C_{3\alpha} x_0+M_3 \right)+b\left(-2 C_{3\alpha} x_0z_d+C_{5\alpha} g x_0^3+C_{6\alpha} q x_0^2 \right) \right. \nonumber \\
&+\left. C_{1\alpha} g^2 x_0^5 +C_{2\alpha} q^2 x_0^3 +C_{3\alpha} x_0 z_d^2 +C_{4\alpha} g q x_0^4-C_{5\alpha} g x_0^3 z_d-C_{6\alpha} q x_0^2 z_d \right]\nonumber
\end{align}
and therefore the optimal controls are given by (\ref{boptim}) and
(\ref{balfaoptim}), respectively since the second derivative is positive in both cases. The formulas (\ref{J3boptim}) and (\ref{J3balfaoptim}) are derived from evaluating $J_3$ and $J_{3\alpha}$ at $b_{op}$ and $b_{\alpha_{op}}$. 

\medskip

\noindent iv) For the distributed-boundary optimal control problems
(\ref{PControl}) and (\ref{PControlalfa}), the functional $J_{4}$
and $J_{4\alpha}$ can be written as:
\begin{align*}
J_4(g,q)&=\dfrac{y_0}{2}\left[ g^2 \left( C_1 x_0^5+M_4 x_0\right)+q^2\left( C_2 x_0^3+M_5\right)+C_4 g q x_0^4\right.\nonumber \\
&+\left. C_5 g x_0^3 (b-z_d)+C_6 q x_0^2 (b-z_d)+C_3 x_0 (b-z_d)^2\right]
\end{align*}
and
\begin{align*}
J_{4\alpha}(g,q)&=\dfrac{y_0}{2}\left[ g^2 \left( C_{1\alpha} x_0^5+M_4 x_0\right)+q^2\left( C_{2\alpha} x_0^3+M_5\right)+C_{4\alpha} g q x_0^4\right.\nonumber \\
&+\left. C_{5\alpha} g x_0^3 (b-z_d)+C_{6\alpha} q x_0^2 (b-z_d)+C_{3\alpha} x_0 (b-z_d)^2\right].
\end{align*}

Therefore, the optimal solutions of the problems (\ref{PControl})
and (\ref{PControlalfa}), take the form (\ref{gqoptim}) and
(\ref{gqalfaoptim}), respectively, due to the second partial derivative test. In addition, the optimal optimal values given by formulas (\ref{J4gqoptim}) and (\ref{J4afagqalfaoptim}) are deduced by evaluating $J_4$ at $(g,q)_{op}$ and $J_{4\alpha}$ at $(g,q)_{\alpha_{op}}$.

\medskip

\noindent v) The convergences can be easily proved by taking into account Remark \ref{convergenciaCi} and the closed forms of the optimal controls and optimal values given by the preceding items (i)-(iv). 
Moreover, the following limits can be computed for the optimal controls:
\begin{align*}
\lim\limits_{\alpha \to \infty} \alpha \vert g_{\alpha_{op}}-g_{op}\vert&=\frac{5 x_0 \Bigl| -150 M_1 q x_0 +4 (45 M_1 -2) ( b-z_d) x_0^4 + 5 q x_0^5  \Bigl| }{4\left(15 M_1+2 x_0^4\right)^2 }\\
\lim\limits_{\alpha \to \infty} \alpha \vert q_{\alpha_{op}}-q_{op}\vert&=\frac{x_0 \Bigl| 60 g M_2 x_0^2 + 5 g x_0^5 + 12(6 M_2  -  x_0^3) (b - z_d)\Bigl|}{8 \left(3 M_2+ x_0^3\right)^2}\\
\lim\limits_{\alpha \to \infty} \alpha \vert b_{\alpha_{op}}-b_{op}\vert&=\frac{x_0 \vert q-g x_0\vert}{ M_3+x_0}
\end{align*}
and for the simultaneous control we have:
\begin{align*}
\lim\limits_{\alpha \to \infty} \alpha \vert g_{\alpha}^{op}-g^{op}\vert&=\frac{40 x_0 (b-z_d)}{\mathcal{P}_1} \Bigl| -207360 M_4 M_5^2-8640 M_4 M_5 x_0^3-1440 M_4 x_0^6+18432 M_5^2 x_0^4\\
&+168 M_5 x_0^7+3 x_0^{10}\Bigl|\\
\lim\limits_{\alpha \to \infty} \alpha \vert q_{\alpha}^{op}-q^{op}\vert&=\frac{8 x_0 (b-z_d)}{\mathcal{P}_1 } \Bigl| 1036800 M_4^2 M_5-172800 M_4^2 x_0^3-227520 M_4 M_5 x_0^4-7080 M_4  x_0^7\\
&-768 M_5 x_0^8+3 x_0^{11} \Bigl|
\end{align*}
with
\begin{align*}
\mathcal{P}_1=3 \left(2880 M_4 M_5+960 M_4 x_0^3+384 M_5 x_0^4+3 x_0^7\right) \left(320 M_4 \left(3 M_5+x_0^3\right)+128 M_5 x_0^4+x_0^7\right)
\end{align*}
 In the case of the optimal values, we have:
\begin{align*}
\lim\limits_{\alpha \to \infty} \alpha  \Big| J_{1\alpha}( g_{\alpha_{op}})-J_1(g_{op}) \Big| &=\frac{x_0 y_0}{192 \left(15 M_1+2 x_0^4\right)^2} \Big| \Big(40 (b-z_d) x_0^3+3q( 40 M_1+3 x_0^4)\Big) \\
& \Big( 8(b-z_d) (45 M_1+x_0^4) +qx_0 (x_0^4-180 M_1) \Big) \Big| \\[0.2cm]
\lim\limits_{\alpha \to \infty} \alpha  \Big| J_{2\alpha}( q_{\alpha_{op}})-J_2(q_{op}) \Big|&=\frac{x_0^2 y_0}{128 \left(3 M_2+x_0^3\right)^2}   \Big| \Big(-4 (b-z_d) x_0+g(8 M_2 +x_0^3)\Big) \\
& \Big( 12(b-z_d) (x_0^3+12M_2)+ gx_0^2 (48M_2+x_0^3)   \Big)\Big| \\[0.2cm]
\Big| J_{3\alpha}( b_{\alpha_{op}})-J_3(b_{op})\Big| & =\frac{1}{\alpha}\frac{ \Big| M_3 x_0 y_0 (g x_0-q) \left(2 g x_0^2-3 q x_0-6 z_d\right) \Big|}{6 (M_3+x_0)}\\[0.2cm]
\lim\limits_{\alpha \to \infty} \alpha  \Big| J_{4\alpha}( g_{\alpha}^{op}, q_{\alpha}^{op})-J_4(g^{op},q^{op}) \Big|&=\frac{64 x_0^3 y_0 (b-z_d)^2 \Big(120 M_4+80 M_5 x_0+x_0^4\Big)}{3 \Big(960 M_4 M_5+320 M_4 x_0^3+128 M_5 x_0^4+x_0^7\Big)^2} \\
& \Big(180 M_4 M_5+15 M_4 x_0^3+4 M_5 x_0^4+x_0^7\Big)
\end{align*}
\end{proof}

\section{Optimal solutions for an annulus in $\mathbb{R}^{2}$}

We consider the following particular domain
\[
\Omega_{2}=\lbrace (r,\theta)\in \mathbb{R}^{2}:
r_1<r<r_2,\,\,\,0\leq\theta<2\pi\rbrace
\]
with boundary $\Gamma_{1}$ and $\Gamma_{2}$ given by (see Figure 2):
\[
\Gamma_{1}=\{(r,\theta)\in\mathbb{R}^{2}: r=r_{1}, \,\, 0\leq \theta
< 2\pi\}
\]
\[
\Gamma_{2}=\{(r,\theta)\in\mathbb{R}^{2}: r=r_{2}, \,\, 0\leq \theta
< 2\pi\}
\]

 \begin{center}
    \includegraphics[scale=0.4]{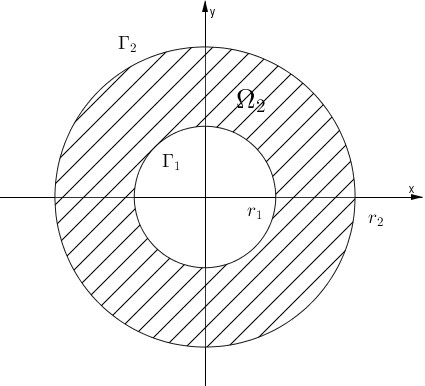}

{\small  Figure 2}
 \end{center}

In similar way to previous Section, if we take constant data $g$,
$b$, $\alpha$, $q$ and the desired system state $z_d\in \mathbb{R}$,
we obtain the following result:
\begin{lem}
i) The system state and the adjoint state for the problem (\ref{P})
are given by
\begin{align*}
u(r,\theta)&=u(r)= g\tfrac{r_1^2}{2}\left(\left(\tfrac{r_2}{r_1} \right)^2\log\left(\tfrac{r}{r_1} \right)-\tfrac{1}{2} \left(\tfrac{r}{r_1} \right)^2+\tfrac{1}{2} \right)-q r_2 \log\left( \tfrac{r}{r_1}\right)+b\\
p(r,\theta)&=p(r)= g   \tfrac{r_1^2 r^2}{8} \left(\tfrac{1}{8}\left( \tfrac{r}{r_1}\right)^2-\tfrac{1}{2}-\left(\tfrac{r_2}{r_1} \right)^2 \left( \log\left(\tfrac{r}{r_1} \right)-1\right) \right)\\
&+q \tfrac{r_2  r^2}{4}  \left( \log\left(\tfrac{r}{r_1} \right)-1 \right)-(b-z_d)\tfrac{r^2}{4}+D_1 \log\left(\tfrac{r}{r_1}\right)+D_2
\end{align*}
where
\begin{align*}
D_1 &= \tfrac{r_2^2}{2}\left[g\tfrac{r_1^2}{2}\left( \left(\tfrac{r_2}{r_1} \right)^2 \left(\log\left( \tfrac{r_2}{r_1}\right)-\tfrac{3}{4} \right)+\tfrac{1}{2} \right)-q r_2 \left(\log\left(\tfrac{r_2}{r_1} \right)-\tfrac{1}{2} \right)  +(b-z_d) \right]\\
D_2&=\tfrac{r_1^2}{4} \left[g \tfrac{r_1^2}{2}\left(\tfrac{3}{8}-\left(\tfrac{r_2}{r_1} \right)^2 \right) +qr_2+ (b-z_d)\right].
\end{align*}
ii) The system state and the adjoint state for the problem
(\ref{Palfa}) are given by
\begin{align*}
u_{\alpha }(r,\theta)&=u_{\alpha }(r)=g\tfrac{r_1^2}{2} \left[ \left( \tfrac{r_2}{r_1}\right)^2\left( \log\left(\tfrac{r}{r_1} \right) +\tfrac{1}{\alpha r_1} \right) - \tfrac{1}{2}\left( \tfrac{r}{r_1}\right)^2+\tfrac{1}{2}-\tfrac{1}{\alpha r_1} \right] \\
& -qr_{2}\left(\log\left(\tfrac{r}{r_1}\right)+\tfrac{1}{\alpha r_{1}}\right)+b\\
p_{\alpha }(r,\theta)&= p_{\alpha}(r)=g  \tfrac{r_1^2 r^2}{8} \left[\tfrac{1}{8}\left(\tfrac{r}{r_1} \right)^2 -\tfrac{1}{2}-\left( \tfrac{r_2}{r_1}\right)^2 \left(\log\left( \tfrac{r}{r_1}\right)-1-\tfrac{r_1}{\alpha r_2^2} +\tfrac{1}{\alpha r_1}\right)\right]\\
& +q  \tfrac{r_2 r^2}{4} \left(\log\left(\tfrac{r}{r_1}\right)-1+\tfrac{1}{\alpha r_1}  \right) -(b-z_d)\tfrac{r^2}{4}\\
&+D_{1\alpha } \log\left( \tfrac{r}{r_1}\right)+D_{2\alpha}
\end{align*}
where
\begin{eqnarray*}
D_{1\alpha} &=& \tfrac{r_2^2}{2} \left[ g\tfrac{r_1^2}{2}\left( \left(\tfrac{r_2}{r_1} \right)^2 \left(\log\left( \tfrac{r_2}{r_1}\right)-\tfrac{3}{4}+\tfrac{1}{\alpha r_1} \right)+\tfrac{1}{2}-\tfrac{1}{\alpha r_1} \right)\right.\\
&-&\left.q r_2 \left(\log\left(\tfrac{r_2}{r_1} \right)-\tfrac{1}{2}+\tfrac{1}{\alpha r_1} \right)+(b-z_d) \right] \\
D_{2\alpha} &=& \tfrac{r_1^2}{4}\left[ g\tfrac{r_1^2}{2}\left( \tfrac{3}{8}-\left(\tfrac{r_2}{r_1} \right)^2 \left( 1+\tfrac{r_1}{\alpha r_2^2}-\tfrac{2}{\alpha r_1}+\tfrac{2}{\alpha^2 r_1^2}\right)+\tfrac{2}{\alpha^2 r_1^2}-\tfrac{1}{2\alpha r_1}   \right)\right.\\
&+&\left. qr_2\left(1-\tfrac{2}{\alpha r_1 }+\tfrac{2}{\alpha^2 r_1^2} \right)+ (b-z_d)\left(1-\tfrac{2}{\alpha r_1} \right)\right] +\tfrac{D_{1\alpha}}{\alpha r_1}.
\end{eqnarray*}

\end{lem}

\begin{rem}\label{convergenciaD}
From the formulas given above, it is clear that $u_{\alpha}$ converges to $u$ and $p_{\alpha}$ to $p$, when $\alpha\to\infty$.
Furthermore, we can prove that there exists a positive constant $K_2=K_2(r_1,r_2,g,q)$ such that:
$$\vert\vert u_{\alpha}-u\vert\vert_{H^1(\Omega_2)}=\vert\vert u_{\alpha}-u\vert\vert_{L^2(\Omega_2)}=\dfrac{K_2}{\alpha}$$
where
$$K_2=\frac{\sqrt{\pi} \left(r_2^2-r_1^2\right)^{1/2} \lvert2 q r_2- g \left(r_2^2-r_1^2\right)\rvert}{2  r_1}.$$
In the same way, a similar estimate can be obtained for the adjoint states $p_{\alpha}$ and $p$. In Appendix A, it is proved that there exists a positive constant $L_2=L_2(r_1,r_2,g,q,b,z_d)$
such that:
\begin{align*}
\lim\limits_{\alpha\to\infty} \alpha \vert\vert p_{\alpha}-p\vert\vert_{L^2(\Omega_2)}& =L_2
\end{align*}

\end{rem}

Now, we present the following lemma that will allow us to obtain the explicit solutions for the optimal control
problems  on the annulus in $\mathbb{R}^{2}$.

\begin{lem}\label{normacorona2}

i) For the problem (\ref{P}), it can be obtained that:
\begin{align*}
\frac{1}{2}  \Vert u-z_d\Vert^2_{L^2(\Omega_2)}&=\,\pi \left[ E_1  g^2 r_1^6 +E_2  q^2  r_1^4 +E_3 r_1^2 (b-z_d)^2+ E_4 g q  r_1^5  +E_5  g r_1^4 (b-z_d)+E_6  q r_1^3 (b-z_d)\right]
\end{align*}
with:
\begin{align*}
E_1&=\tfrac{1}{8} \left[-\tfrac{1}{12}+\tfrac{5}{8} \left(\tfrac{r_2}{r_1}\right)^2+\left(\tfrac{r_2}{r_1}\right)^4 \left(\log \left(\tfrac{r_2}{r_1}\right)-\tfrac{5}{4}\right)+\left(\tfrac{r_2}{r_1}\right)^6 \left(\log ^2\left(\tfrac{r_2}{r_1}\right)-\tfrac{3}{2} \log \left(\tfrac{r_2}{r_1}\right)+\tfrac{17}{24}\right)\right]\\
E_2&=\tfrac{1}{4} \left(-\left(\tfrac{r_2}{r_1}\right)^2+\left(\tfrac{r_2}{r_1}\right)^4 \left(2 \log ^2\left(\tfrac{r_2}{r_1}\right)-2 \log \left(\tfrac{r_2}{r_1}\right)+1\right)\right)\\
E_3&=\tfrac{1}{2} \left(-1+\left(\tfrac{r_2}{r_1}\right)^2\right)\\
E_4&=\tfrac{1}{4} \left[-\tfrac{3 }{8 }\left( \tfrac{r_2}{r_1}\right)+\left(\tfrac{r_2}{r_1}\right)^3 \left(\tfrac{3}{2}-\log \left(\tfrac{r_2}{r_1}\right)\right)+\left(\tfrac{r_2}{r_1}\right)^5 \left(-2 \log ^2\left(\tfrac{r_2}{r_1}\right)+\tfrac{5}{2} \log \left(\tfrac{r_2}{r_1}\right)-\tfrac{9}{8}\right)\right]\\
E_5&=\tfrac{1}{2} \left(-\tfrac{1}{4}+\left(\tfrac{r_2}{r_1}\right)^2+\left(\tfrac{r_2}{r_1}\right)^4 \left(\log \left(\tfrac{r_2}{r_1}\right)-\tfrac{3}{4}\right)\right)\\
E_6&=-\left(\tfrac{1}{2}\left(\tfrac{r_2}{r_1}\right)+\left(\tfrac{r_2}{r_1}\right)^3 \left(\log \left(\tfrac{r_2}{r_1}\right)-\tfrac{1}{2}\right)\right)
\end{align*}

ii) For the problem (\ref{Palfa}), we have:
\begin{align*}
\frac{1}{2}  \Vert u-z_d\Vert^2_{L^2(\Omega_2)}&=\,\pi \left[ E_{1\alpha}  g^2 r_1^6 +E_{2\alpha}  q^2 r_1^4 +E_{3\alpha} r_1^2(b-z_d)^2+\right.\\
&\left.+E_{4\alpha}  g q r_1^5 +E_{5\alpha} g   r_1^4 (b-z_d)+E_{6\alpha}  q r_1^3(b-z_d)\right]
\end{align*}
with
\begin{align*}
E_{1\alpha}&= \tfrac{1}{8} \left[ -\tfrac{1}{12}+\tfrac{1}{2\alpha r_1}-\tfrac{1}{\alpha^2 r_1^2}+\left(\tfrac{r_2}{r_1} \right)^2 \left( \tfrac{5}{8}-\tfrac{5}{2\alpha r_1}+\tfrac{3}{\alpha^2 r_1^2}\right)\right. \\
&+\left(\tfrac{r_2}{r_1} \right)^4 \left(\log\left(\tfrac{r_2}{r_1} \right)-\tfrac{5}{4}+\tfrac{1}{\alpha r_1}\left(\tfrac{7}{2}-2\log\left(\tfrac{r_2}{r_1} \right) \right)-\tfrac{3}{\alpha^2 r_1^2} \right)\\
&+\left. \left(\tfrac{r_2}{r_1} \right)^6 \left( \log^2\left(\tfrac{r_2}{r_1} \right)-\tfrac{3}{2}\log\left(\tfrac{r_2}{r_1} \right)+\tfrac{17}{24}-\tfrac{3}{2\alpha r_1}+\tfrac{2}{\alpha r_1}\log\left(\tfrac{r_2}{r_1} \right)+\tfrac{1}{\alpha^2 r_1^2}\right)\right]\\
E_{2\alpha} &=\tfrac{1}{4}\left[ -\left(\tfrac{r_2}{r_1} \right)^2\left( 1+\tfrac{2}{\alpha^2 r_1^2}-\tfrac{2}{\alpha r_1}\right)\right.\\
&\left. + \left(\tfrac{r_2}{r_1} \right)^4 \left(2\log^2\left(\tfrac{r_2}{r_1} \right)-2\log\left(\tfrac{r_2}{r_1} \right)+1+\tfrac{2}{\alpha^2 r_1^2}+\tfrac{1}{\alpha r_1}\left( 4\log\left(\tfrac{r_2}{r_1} \right)-2\right) \right)\right]\\
E_{3\alpha}&= E_3=\tfrac{1}{2}\left( -1+\left(\tfrac{r_2}{r_1} \right)^2\right)\\
E_{4\alpha}&= \tfrac{1}{4}\left[\tfrac{r_2}{r_1} \left(-\tfrac{3}{8}+\tfrac{3}{2\alpha r_1}-\tfrac{2}{\alpha^2 r_1^2} \right) \right.\\
&+ \left(\tfrac{r_2}{r_1} \right)^3 \left( \tfrac{3}{2}-\log\left(\tfrac{r_2}{r_1} \right)+\tfrac{1}{\alpha r_1} \left( 2\log\left(\tfrac{r_2}{r_1} \right)-4\right)+\tfrac{4}{\alpha^2 r_1^2}\right)\\
&+\left. \left(\tfrac{r_2}{r_1} \right)^5 \left( -2\log^2\left(\tfrac{r_2}{r_1} \right)+\tfrac{5}{2}\log\left(\tfrac{r_2}{r_1} \right)-\tfrac{9}{8}-\tfrac{4}{\alpha r_1}\log\left(\tfrac{r_2}{r_1} \right)+\tfrac{5}{2\alpha r_1} -\tfrac{2}{\alpha^2 r_1^2}\right) \right]\\
E_{5\alpha} &= \tfrac{1}{2}\left[ -\tfrac{1}{4}+\tfrac{1}{\alpha r_1}+\left(\tfrac{r_2}{r_1} \right)^2\left( 1-\tfrac{2}{\alpha r_1}\right)+\left(\tfrac{r_2}{r_1} \right)^4 \left(\log\left(\tfrac{r_2}{r_1} \right)-\tfrac{3}{4}+\tfrac{1}{\alpha r_1} \right)\right]\\
E_{6\alpha} &= - \left( \tfrac{r_2}{r_1}\left(\tfrac{1}{2}-\tfrac{1}{\alpha r_1} \right)+\left(\tfrac{r_2}{r_1} \right)^3 \left( \log\left(\tfrac{r_2}{r_1} \right)-\tfrac{1}{2}+\tfrac{1}{\alpha r_1}\right)  \right).
\end{align*}

\end{lem}

\begin{rem}\label{ConvergenciaEi}
It is immediate that $E_{i\alpha}$ converges to $E_i$, when \mbox{$\alpha\to\infty$} for $i=1,2,\dots,6$.
\end{rem}

\begin{thm}
i) For the distributed optimal control problems (\ref{PControlJ1}) and
(\ref{PControlalfaJ1}), the optimal solutions are given by:
\begin{equation}\label{goptimocorona2}
g_{op}=-\frac{E_4 q r_1+E_5 (b-z_d)}{2 r_1^2 \left(E_1+E_3\frac{ M_1}{r_1^4}\right)}
\end{equation}
and
\begin{equation}\label{goptimoalfacorona2}
g_{\alpha_{op}}=-\frac{E_{4\alpha} q r_1+E_{5\alpha} (b-z_d)}{2 r_1^2 \left(E_{1\alpha}+E_{3\alpha}\frac{ M_1}{r_1^4}\right)}
\end{equation}
and the optimal values can be expressed as:

\begin{align}\label{J1gopCorona2}
J_1(g_{op})=\tfrac{\pi  r_1^2 \left[4 \left(E_1+E_3\tfrac{ M_1}{r_1^4}\right) \left(E_2 q^2 r_1^2+ E_3 (b-z_d)^2+E_6 q r_1 (b-z_d)\right)-\left(E_4 q r_1+E_5 (b-z_d)\right)^2\right]}{4 \left(E_1+E_3\tfrac{ M_1}{r_1^4}\right)}
\end{align}
and
\begin{align}\label{J1alfagalfaopCorona2}
J_{1\alpha}(g_{\alpha_{op}})=\tfrac{\pi  r_1^2 \left[4 \left(E_{1\alpha}+E_{3\alpha}\tfrac{ M_1}{r_1^4}\right) \left(E_{2\alpha} q^2 r_1^2+E_{3\alpha} (b-z_d)^2+E_{6\alpha} q r_1 (b-z_d)\right)-\left(E_{4\alpha} q r_1+ E_{5\alpha} (b-z_d)\right)^2\right]}{4 \left(E_{1\alpha}+E_{3\alpha}\tfrac{ M_1}{r_1^4}\right)}
\end{align}

\noindent ii) For the boundary optimal control problems (\ref{PControlJ2}) and
(\ref{PControlalfaJ2}), the optimal solutions are given by:
\begin{equation}\label{qoptimocorona2}
q_{op}=-\frac{E_4 g r_1^2+ E_6 (b-z_d)}{2 r_1 \left(E_2+\tfrac{M_2 r_2}{r_1^4}\right)}
\end{equation}
and
\begin{equation}\label{qoptimoalfacorona2}
q_{\alpha_{op}}=-\frac{E_{4\alpha} g r_1^2+ E_{6\alpha} (b-z_d)}{2 r_1 \left(E_{2\alpha}+\tfrac{M_2 r_2}{r_1^4}\right)}
\end{equation}
where the optimal values are given by:
\begin{align}
J_2(q_{op})&=\tfrac{\pi  r_1^2 \left[4 \left(E_2+\tfrac{M_2 r_2}{ r_1^4}\right) \left(E_1 g^2 r_1^4+E_3 (b-z_d)^2+E_5 g r_1^2 (b-z_d)\right)-\left(E_4 g r_1^2+E_6 (b-z_d)\right)^2\right]}{4 \left(E_2+\tfrac{M_2 r_2}{ r_1^4}\right)}\label{J2qopCorona2}
\end{align}
and
\begin{align}
J_{2\alpha}(q_{\alpha_{op}})&=\tfrac{\pi  r_1^2 \left[4 \left(E_{2\alpha}+\tfrac{M_2 r_2}{ r_1^4}\right) \left(E_{1\alpha} g^2 r_1^4+ E_{3\alpha} (b-z_d)^2+E_{5\alpha} g r_1^2 (b-z_d)\right)-\left(E_{4\alpha} g r_1^2+E_{6\alpha} (b-z_d)\right)^2\right]}{4 \left(E_{2\alpha}+\tfrac{M_2 r_2}{ r_1^4}\right)}.\label{J2alfaqalfaopCorona2}
\end{align}

\noindent iii) For the boundary optimal control problems (\ref{PControlJ3}) and
(\ref{PControlalfaJ3}), the optimal controls are given by
\begin{equation}\label{boptimocorona2}
b_{op}=-\frac{E_5 g r_1^2+E_6 q r_1-2 E_3 z_d}{2 \left(E_3+\frac{M_3}{r_1}\right)}
\end{equation}
and
\begin{equation}\label{boptimoalfacorona2}
b_{\alpha_{op}}=-\frac{E_{5\alpha} g r_1^2+E_{6\alpha} q r_1-2 E_{3\alpha} z_d}{2 \left(E_{3\alpha}+\frac{M_3}{r_1}\right)}
\end{equation}
respectively. In addition, the optimal values are given by:
\begin{align}
J_3(b_{op})&= \left[4 \left(E_3+\tfrac{M_3}{r_1}\right) \left(E_1 g^2 r_1^4+E_2 q^2 r_1^2+E_3 z_d^2+E_4 g q r_1^3-E_5 g r_1^2 z_d-E_6 q r_1 z_d\right)\right.\nonumber\\
&\left.-\left(-2 E_3 z_d+E_5 g r_1^2+E_6 q r_1\right)^2\right]\tfrac{\pi  r_1^2}{4 \left(E_3+\tfrac{M_3}{r_1}\right)}\label{J3bopCorona2}
\end{align}
and
\begin{small}
\begin{align}
J_{3\alpha}(b_{\alpha_{op}})&= \left[4 \left(E_{3\alpha}+\tfrac{M_3}{r_1}\right) \left(E_{1\alpha} g^2 r_1^4+E_{2\alpha} q^2 r_1^2+E_{3\alpha} z_d^2+E_{4\alpha} g q r_1^3-E_{5\alpha} g r_1^2 z_d-E_{6\alpha} q r_1 z_d\right)\right.
\nonumber\\
&\left.-\left(-2 E_{3\alpha} z_d+E_{5\alpha} g r_1^2+E_{6\alpha} q r_1\right)^2\right]\tfrac{\pi  r_1^2}{4 \left(E_{3\alpha}+\frac{M_3}{r_1}\right)}.\label{J3alfabalfaopCorona2}
\end{align}
\end{small}

\noindent iv) For the distributed-boundary optimal control problem
(\ref{PControl}) and (\ref{PControlalfa}), the optimal solutions are given by
\begin{align}
(g,q)_{op}=(g^{op},q^{op})=& \left(\frac{(b-z_d)}{r_1^2} \Delta_2,\frac{(b-z_d)}{r_1}\Pi_2  \right)\label{bicontrolcorona2}
\end{align}
with
$$\Delta_2=\tfrac{E_4 E_6-2 E_5 \left(E_2+\tfrac{M_5 r_2}{r_1^4}\right)}{4 \left(E_1+E_3\tfrac{ M_4}{r_1^4}\right) \left(E_2+\tfrac{M_5 r_2}{r_1^4}\right)-E_4^2},\qquad \Pi_2=\tfrac{E_4 E_5-2 E_6 \left(E_1+E_3\tfrac{ M_4}{r_1^4}\right)}{ 4 \left(E_1+E_3\tfrac{ M_4}{r_1^4}\right) \left(E_2+\frac{M_5 r_2}{r_1^4}\right)-E_4^2}$$
and
\begin{align}
(g,q)_{\alpha_{op}}=(g_{\alpha}^{op},q_{\alpha}^{op})=&\left(\frac{(b-z_d)}{r_1^2}  \Delta_{2_\alpha} , \frac{(b-z_d)}{r_1} \Pi_{2_\alpha} \right)\label{bicontrolalfacorona2}
\end{align}
where
$$\Delta_{2_\alpha}=\tfrac{E_{4\alpha} E_{6\alpha}-2 E_{5\alpha} \left(E_{2\alpha}+\tfrac{M_5 r_2}{r_1^4}\right)}{4 \left(E_{1\alpha}+E_{3\alpha}\tfrac{ M_4}{r_1^4}\right) \left(E_{2\alpha}+\tfrac{M_5 r_2}{r_1^4}\right)-E_{4\alpha}^2},\quad  \Pi_{2_\alpha}=\tfrac{E_{4\alpha} E_{5\alpha}-2 E_{6\alpha} \left(E_{1\alpha}+E_{3\alpha}\tfrac{ M_4}{r_1^4}\right)}{4 \left(E_{1\alpha}+E_{3\alpha}\tfrac{ M_4}{r_1^4}\right) \left(E_{2\alpha}+\frac{M_5 r_2}{r_1^4}\right)-E_{4\alpha}^2}$$
Moreover, the optimal values are given by
\begin{small}
\begin{align}
J_4(g^{op},q^{op})&=\frac{\pi  r_1^2 (b-z_d)^2}{\left(4 \left(E_1+E_3\tfrac{ M_4}{r_1^4}\right) \left(E_2+\tfrac{M_5 r_2}{r_1^4}\right)-E_4^2\right)} \left[4 E_3 \left(E_1+E_3\tfrac{ M_4}{r_1^4}\right) \left(E_2+\tfrac{M_5 r_2}{r_1^4}\right)\right.\nonumber\\
&\left.-E_6^2 \left(E_1+E_3\tfrac{ M_4}{r_1^4}\right)-E_5^2 \left(E_2+\tfrac{M_5 r_2}{r_1^4}\right)-E_3 E_4^2+E_4 E_5 E_6\right]\label{J4gqopCorona2}
\end{align}
\end{small}
and
\begin{small}
\begin{align}
J_{4\alpha}(g_{\alpha}^{op},q_{\alpha}^{op})&=\frac{\pi  r_1^2 (b-z_d)^2}{\left(4 \left(E_{1\alpha}+E_{3\alpha}\tfrac{ M_4}{r_1^4}\right) \left(E_{2\alpha}+\tfrac{M_5 r_2}{r_1^4}\right)-E_{4\alpha}^2\right)} \left[4 E_{3\alpha} \left(E_{1\alpha}+E_{3\alpha}\tfrac{ M_4}{r_1^4}\right) \left(E_{2\alpha}+\tfrac{M_5 r_2}{r_1^4}\right)\right.\nonumber\\
&\left.-E_{6\alpha}^2 \left(E_{1\alpha}+E_{3\alpha}\tfrac{ M_4}{r_1^4}\right)-E_{5\alpha}^2 \left(E_{2\alpha}+\tfrac{M_5 r_2}{r_1^4}\right)-E_{3\alpha} E_{4\alpha}^2+E_{4\alpha} E_{5\alpha} E_{6\alpha}\right]\label{J4alfagqalfaopCorona2}
\end{align}
\end{small}

\noindent v) The convergences and estimates obtained in (v) of Theorem \ref{TeoremaRectangulo}  also hold for the annulus in $\mathbb{R}^2$.

\end{thm}

\begin{proof}

i) Taking into account that the functional $J_1$ and $J_{1\alpha}$ can be expressed in the following quadratic forms:
\begin{align*}
J_1(g)&=\pi \left[ g^2 (E_1 r_1^6+M_1 E_3 r_1^2) +g \left(E_4  q r_1^5+E_5(b-z_d) r_1^4\right)\right.\\
&+\left.\left(E_2 q^2 r_1^4+E_3(b-z_d)^2 r_1^2+E_6 q r_1^2(b-z_d)\right) \right]
\end{align*}
and
\begin{align*}
J_{1\alpha}(g)&=\pi \left[ g^2\left(E_{1\alpha} r_1^6+M_1 E_{3\alpha} r_1^2\right) +g\left(E_{4\alpha}qr_1^5+E_{5\alpha}(b-z_d) r_1^4\right)\right.\\
&\left.+\left(E_{2\alpha} q^2 r_1^4+E_{3\alpha}(b-z_d)^2 r_1^2+E_{6\alpha} q r_1^2(b-z_d)\right) \right]
\end{align*}
it can be  obtained that the optimal solutions $g_{op}$ and $g_{\alpha_{op}}$ for the
problems (\ref{PControlJ1}) and (\ref{PControlalfaJ1}) are given by
(\ref{goptimocorona2}) and (\ref{goptimoalfacorona2}), respectively, since the second derivative is positive in both cases. The optimal values formulas (\ref{J1gopCorona2}) and (\ref{J1alfagalfaopCorona2}) are deduced by evaluating $J_1$ and $J_{1\alpha}$ at $g_{op}$ and $g_{\alpha_{op}}$, respectively.

\medskip
\noindent ii) The functional $J_2$ and $J_{2\alpha}$ are given by the expressions:
\begin{align*}
J_2(q)&= \pi \left[ q^2 \left(E_2 r_1^4+M_2 r_2\right)+q\left( E_4 r_1^5 g+E_6 r_1^3(b-z_d)\right)\right.\\
&\left.+\left( E_1 r_1^6 g^2+E_3 r_1^2 (b-z_d)^2+E_5 r_1^4 g (b-z_d)\right)\right]
\end{align*}
and
\begin{align*}
J_{2\alpha}(q)&= \pi \left[ q^2\left(E_{2\alpha}r_1^4+M_2 r_2\right)+q \left( E_{4\alpha} r_1^5 g+E_{6\alpha} r_1^3(b-z_d)\right)\right.\\
&\left.+\left( E_{1\alpha} r_1^6 g^2+E_{3\alpha} r_1^2(b-z_d)^2+E_{5\alpha} r_1^4 g (b-z_d)\right)\right].
\end{align*}
Therefore it is immediate that the optimal controls for problems (\ref{PControlJ2}) and (\ref{PControlalfaJ2}) are given by
(\ref{qoptimocorona2}) and (\ref{qoptimoalfacorona2}), respectively, since the second derivative is positive in both cases.

The computation of  $J_2(q_{op})$ and $J_{2\alpha}(q_{\alpha_{op}})$ leads to the closed formulas (\ref{J2qopCorona2}) and (\ref{J2alfaqalfaopCorona2}) for the optimal values of the control problems.

\medskip
\noindent iii)
For the problems (\ref{PControlJ3}) and (\ref{PControlalfaJ3}), the
functional $J_{3}$ and $J_{3\alpha}$ are given by
\begin{align*}
J_3(b)&= \pi \left[ \left(E_3 r_1^2+M_3 r_1\right)b^2+\left( -2z_d E_3 r_1^2+E_5 r_1^4 g+E_6 r_1^3 q\right)b\right.\\
&\left.+\left( E_1 r_1^6 g^2+E_2 r_1^4 q^2+E_3 r_1^2 z_d^2 +E_4 r_1^5 g q -E_5 r_1^4 g z_d-E_6 r_1^3 q z_d\right)\right]
\end{align*}
and
\begin{align*}
J_{3\alpha}(b)&= \pi \left[ \left(E_{3\alpha} r_1^2+M_3 r_1\right)b^2+\left( -2z_d E_{3\alpha}r_1^2+E_{5\alpha}r_1^4 g+E_{6\alpha}r_1^3 q\right)b\right.\\
&\left.+\left( E_{1\alpha}r_1^6 g^2+E_{2\alpha}r_1^4 q^2+E_{3\alpha} r_1^2 z_d^2 +E_{4\alpha} r_1^5 g q -E_{5\alpha} r_1^4 g z_d-E_{6\alpha} r_1^3 q z_d\right)\right].
\end{align*}
Therefore the optimal controls are given by
(\ref{boptimocorona2}) and (\ref{boptimoalfacorona2}), respectively, since the second derivative is positive in both cases.

The optimal values given by expressions (\ref{J3bopCorona2}) and (\ref{J3alfabalfaopCorona2}) are obtained by computing $J_3$ and $J_{3\alpha}$ at $b_{op}$ and $b_{\alpha_{op}}$, respectively.
\medskip

\noindent iv)
For the distributed-boundary optimal control problems
(\ref{PControl}) and (\ref{PControlalfa}), the functional $J_{4}$
can be expressed as
\begin{align*}
J_{4}(g,q)&= \pi\left[ (E_1r_1^6+M_4E_3 r_1^2)g^2+(E_2r_1^4+M_5 r_2)q^2+E_4 r_1^5 gq\right.\\
&+E_5 r_1^4g(b-z_d)+\left. E_6 r_1^3q (b-z_d)+E_3r_1^2(b-z_d)^2\right]
\end{align*}
and the functional $J_{4\alpha}$ is given by:
\begin{align*}
J_{4\alpha}(g,q)&= \pi\left[ (E_{1\alpha}r_1^6+M_4E_{3\alpha}r_1^2)g^2+(E_{2\alpha}r_1^4+M_5 r_2)q^2+E_{4\alpha}r_1^5 gq\right.\\
&+\left.E_{5\alpha} r_1^4 g(b-z_d) +E_{6\alpha} r_1^3 q (b-z_d)+E_{3\alpha} r_1^2(b-z_d)^2\right]
\end{align*}
from where it can be obtained that the optimal
solutions are given by (\ref{bicontrolcorona2}) and
(\ref{bicontrolalfacorona2}), respectively, due to the second partial derivative test.
Formulas (\ref{J4gqopCorona2}) and (\ref{J4alfagqalfaopCorona2}) are deduced by evaluating $J_4$ at $(g,q)_{op}$ and $J_{4\alpha}$ at $(g,q)_{\alpha_{op}}$.

\medskip
\noindent v)
The convergences and estimates of the optimal controls and the optimal values when $\alpha\to \infty$ are obtained  by taking into account the closed formulas given in (i)-(iv) and the Remark \ref{ConvergenciaEi}. As the computations  become cumbersome, they can be found in the Appendix A.

\end{proof}

\section{Optimal solutions for a spherical shell in $\mathbb{R}^{3}$}

We consider the particular domain
\[
\Omega_{3}=\lbrace (r,\theta,\phi):
r_1<r<r_2;\,0\leq \theta<2\pi;\,\,0\leq\phi\leq \pi\rbrace
\] with boundary $\displaystyle\Gamma= \cup_{i=1}^{2} \Gamma_{i}$, where
\[
\Gamma_{1}=\{(r_1,\theta,\phi)\in\mathbb{R}^{3}: \,\, 0\leq \theta
< 2\pi\, , \, 0\leq \phi\leq\pi\}
\]
\[
\Gamma_{2}=\{(r_2,\theta,\phi)\in\mathbb{R}^{3}: \,\, 0\leq \theta
< 2\pi\, , \, 0\leq\phi\leq\pi\}.
\]

In similar way to previous Sections, if we take constant data $g$,
$b$, $\alpha$, $q$ and the desired system state $z_d\in \mathbb{R}$,
we obtain the following result:
\begin{lem}
i) The system state and the adjoint state for the problem (\ref{P})
are given by
\begin{align*}
u(r,\theta,\phi)=& u(r)=g \frac{r_1^2}{3}\left[\frac{1}{2}-\frac{1}{2}\left( \frac{r}{r_1}\right)^2+\left(\tfrac{r_2}{r_1} \right)^3-\frac{r_2}{r}\left(\tfrac{r_2}{r_1} \right)^2 \right]+q \frac{r_2^2}{r_1} \left( \frac{r_1}{r}-1\right)+ b\\
p(r,\theta,\phi)&= p(r)=g r_1^2 \frac{r^2}{6} \left( \frac{1}{20} \left( \frac{r}{r_1}\right)^2+\left( \frac{r_2}{r}\right)\left( \tfrac{r_2}{r_1}\right)^2-\frac{1}{3}\left( \tfrac{r_2}{r_1}\right)^3-\frac{1}{6}\right)\\
&+qr_2^2 \frac{r}{2}\left(\frac{r}{3r_1}-1 \right)-\frac{r^2}{6}(b-z_d)+\frac{F_1}{r}+F_2
\end{align*}
where
\begin{eqnarray*}
F_1 &=& g r_1^2r_2^3 \left( -\frac{1}{9} \left( \tfrac{r_2}{r_1}\right)^3+\frac{1}{5}\left( \tfrac{r_2}{r_1}\right)^2-\frac{1}{18} \right) + q r_2^4 \left( \frac{1}{3}\left( \tfrac{r_2}{r_1}\right)-\frac{1}{2}\right)-\frac{r_2^3}{3}(b-z_d)\\
F_2&=& g \frac{r_1^4}{9}\left(\frac{7}{40}- \left( \tfrac{r_2}{r_1}\right)^3\right) +q r_1\frac{r_2^2}{3}+\frac{r_1^2}{6}(b-z_d)-\frac{F_1}{r_1}.
\end{eqnarray*}

ii) The system state and the adjoint state for the problem
(\ref{Palfa}) are given by
\begin{align*}
u_{\alpha }(r,\theta,\phi)&=u_{\alpha}(r)= g\frac{r_1^2}{3}\left[\frac{1}{2}-\frac{1}{\alpha r_1}-\frac{1}{2} \left( \frac{r}{r_1}\right)^2+\left( \tfrac{r_2}{r_1}\right)^3 \left(1+\frac{1}{\alpha r_1} \right) -\frac{r_2}{r} \left( \tfrac{r_2}{r_1}\right)^2\right]\\
& +q \frac{r_2^2}{r_1}\left( \frac{r_1}{r}-1-\frac{1}{\alpha r_1}\right)+b\\
p_{\alpha }(r,\theta,\phi)&= p_{\alpha}(r)= g r_1^2\frac{r^2}{6} \left(\frac{1}{20} \left(\frac{r}{r_1} \right)^2+\left(\frac{r_2}{r}\right)\left(\tfrac{r_2}{r_1} \right)^2-\frac{1}{3}\left(\tfrac{r_2}{r_1} \right)^3 \left(1+\frac{1}{\alpha r_1} \right)-\frac{1}{6}+\frac{1}{3\alpha r_1}  \right)\\
&+q r_2^2 \frac{r}{2} \left( \frac{r}{3 r_1}-1+\frac{r}{3\alpha r_1^2}\right)- \frac{r^2}{6}(b-z_d)+\frac{F_{1\alpha}}{r}+F_{2\alpha}
\end{align*}
where
\begin{eqnarray*}
F_{1\alpha} &=& g r_1^2 r_2^3\left(-\frac{1}{9} \left( \tfrac{r_2}{r_1}\right)^3 \left(1+\frac{1}{\alpha r_1} \right)+\frac{1}{5} \left( \tfrac{r_2}{r_1}\right)^2-\frac{1}{18}\left( 1-\frac{2}{\alpha r_1}\right) \right)\\
&+& q r_2^4 \left( \frac{1}{3}\left(\tfrac{r_2}{r_1}\right) \left(1+\frac{1}{\alpha r_1} \right)-\frac{1}{2} \right)-\frac{r_2^3}{3 }(b-z_d)\\
F_{2\alpha} &=& g \frac{r_1^4}{9} \left[\frac{7}{40}-\frac{7}{10\alpha r_1}+\frac{1}{\alpha^2 r_1^2}-\left( \tfrac{r_2}{r_1}\right)^3\left( 1-\frac{1}{\alpha r_1}+\frac{1}{\alpha^2 r_1^2}\right)
 \right]\\
 &+&q r_1\frac{r_2^2}{3} \left( 1-\frac{1}{\alpha r_1}+\frac{1}{\alpha^2 r_1^2}\right)+\frac{r_1^2}{6}(b-z_d)\left(1-\frac{2}{\alpha r_1} \right)-\frac{F_{1\alpha}}{r_1}\left( 1+\frac{1}{\alpha r_1}\right).
\end{eqnarray*}

\end{lem}

\begin{rem}\label{convergenciaF}
The convergences of $u_{\alpha}$ to $u$, and $p_{\alpha}$ to $p$, when $\alpha\to\infty$ can be immediately verified.

In addition, there exists a positive constant $K_3=K_3(r_1,r_2,g,q)$ such that:
$$\vert\vert u_{\alpha}-u\vert\vert_{H^1(\Omega_3)}=\frac{K_3}{\alpha}$$
with 
$$ K_3=\left(\frac{4\pi( r_2^3-r_1^3 ) (3 q r_2^2 + g (r_1^3 - r_2^3))^2}{27  r_1^4}\right)^{1/2}$$
Analogously, a similar estimate can be proved for the adjoint states $p_{\alpha}$ and $p$ (see Appendix A).

\end{rem}

Now, we present the following lemma that will allow us to obtain the explicit solutions for the optimal control
problems  on the spherical
shell in $\mathbb{R}^{3}$.

\begin{lem}\label{normacorona3}

i) For the problem (\ref{P}), it can be obtained that:
\begin{align*}
\frac{1}{2}  \Vert u-z_d\Vert^2_{L^2(\Omega_3)}&=\,\pi \left[ G_1 r_1^7 g^2+ G_2 r_1 r_2^4 q^2+G_3 r_1^3 (b-z_d)^2 +G_4 r_1^4 r_2^2 gq \right.\\
 &+G_5 r_1^5 g (b-z_d)+\left.G_6 r_1^2 r_2^2 q (b-z_d)\right]
\end{align*}
with:
\begin{align*}
G_1&=-\tfrac{2}{945}+\tfrac{1}{45} \left(\tfrac{r_2}{r_1}\right)^3-\tfrac{1}{15}\left(\tfrac{r_2}{r_1}\right)^5+\tfrac{1}{7}\left(\tfrac{r_2}{r_1}\right)^7-\tfrac{2}{15}\left(\tfrac{r_2}{r_1}\right)^8+\tfrac{1}{27} \left(\tfrac{r_2}{r_1}\right)^9\\
G_2&= -\tfrac{1}{3}+\tfrac{r_2}{r_1}-\left(\tfrac{r_2}{r_1}\right)^2+\tfrac{1}{3}\left(\tfrac{r_2}{r_1}\right)^3\\
G_3&=\tfrac{1}{3}\left(-1+\left(\tfrac{r_2}{r_1}\right)^3 \right)\\
G_4&=  -\tfrac{7}{180}+\tfrac{1}{6}\left(\tfrac{r_2}{r_1}\right)^2+\tfrac{1}{9}\left(\tfrac{r_2}{r_1}\right)^3-\tfrac{3}{4}\left(\tfrac{r_2}{r_1}\right)^4+\tfrac{11}{15}\left(\tfrac{r_2}{r_1}\right)^5-\tfrac{2}{9}\left(\tfrac{r_2}{r_1}\right)^6 \\
G_5&=  -\tfrac{2}{45}+\tfrac{2}{9}\left(\tfrac{r_2}{r_1}\right)^3-\tfrac{2}{5}\left(\tfrac{r_2}{r_1}\right)^5+\tfrac{2}{9}\left(\tfrac{r_2}{r_1}\right)^6 \\
G_6&= -\tfrac{1}{3}+\left( \tfrac{r_2}{r_1}\right)^2-\tfrac{2}{3}\left( \tfrac{r_2}{r_1}\right)^3
\end{align*}

ii) For the problem (\ref{Palfa}), we have:
\begin{align*}
\frac{1}{2}  \Vert u_\alpha-z_d\Vert^2_{L^2(\Omega_3)}&= \pi \left[ G_{1\alpha} r_1^7 g^2+ G_{2\alpha} r_1 r_2^4 q^2+G_{3\alpha} r_1^3 (b-z_d)^2 +G_{4\alpha}r_1^4 r_2^2   gq\right. \\
&\left. +G_{5\alpha} r_1^5 g (b-z_d)+G_{6\alpha}  r_1^2 r_2^2q (b-z_d)\right]
\end{align*}
with
\begin{align*}
G_{1\alpha}&= -\tfrac{2}{945}+\tfrac{2}{135\alpha r_1}-\tfrac{1}{27 \alpha^2 r_1^2}+\left(\tfrac{r_2}{r_1}\right)^3\left( \tfrac{1}{45}-\tfrac{4}{45 \alpha r_1}+\tfrac{1}{9\alpha^2 r_1^2}\right) \\
&-\tfrac{1}{15}\left(\tfrac{r_2}{r_1}\right)^5 \left( 1-\tfrac{2}{\alpha r_1}\right)-\tfrac{1}{9\alpha^2 r_1^2}\left(\tfrac{r_2}{r_1}\right)^6 +\tfrac{1}{7}\left(\tfrac{r_2}{r_1}\right)^7-\tfrac{2}{15}\left(\tfrac{r_2}{r_1}\right)^8 \left(1+\tfrac{1}{\alpha r_1} \right)\\
&+\tfrac{1}{27}\left(\tfrac{r_2}{r_1}\right)^9 \left( 1+\tfrac{2}{\alpha r_1}+\tfrac{1}{\alpha^2 r_1^2}\right)\\
G_{2\alpha} &= -\tfrac{1}{3}\left( 1-\tfrac{1}{\alpha r_1}+\tfrac{1}{\alpha^2 r_1^2}\right) +\tfrac{r_2}{r_1}-\left(\tfrac{r_2}{r_1}\right)^2 \left( 1+\tfrac{1}{\alpha r_1}\right)+\tfrac{1}{3}\left(\tfrac{r_2}{r_1}\right)^3 \left(1+\tfrac{2}{\alpha r_1}+\tfrac{1}{\alpha^2 r_1^2} \right)\\
\end{align*}
\begin{align*}
G_{3\alpha}&= G_3=\tfrac{1}{3}\left(-1+\left(\tfrac{r_2}{r_1}\right)^3 \right) \\
G_{4\alpha}&= -\tfrac{7}{180}+\tfrac{7}{45 \alpha r_1}-\tfrac{2}{\alpha^2 r_1^2}+\tfrac{1}{6}\left(\tfrac{r_2}{r_1}\right)^2 \left(1-\tfrac{2}{\alpha r_1} \right) +\tfrac{1}{9} \left(\tfrac{r_2}{r_1}\right)^3 \left( 1-\tfrac{1}{\alpha r_1}+\tfrac{4}{\alpha^2 r_1^2}\right)\\
&-\tfrac{3}{4}\left(\tfrac{r_2}{r_1}\right)^4+\tfrac{11}{15}\left(\tfrac{r_2}{r_1}\right)^5 \left( 1+\tfrac{1}{\alpha r_1}\right)-\tfrac{2}{9}\left(\tfrac{r_2}{r_1}\right)^6 \left(1+\tfrac{2}{\alpha r_1}+\tfrac{1}{\alpha^2 r_1^2} \right)\\
G_{5\alpha} &= -\tfrac{2}{45}+\tfrac{2}{9\alpha r_1}+\tfrac{2}{9}\left(\tfrac{r_2}{r_1}\right)^3\left( 1-\tfrac{2}{\alpha r_1}\right)-\tfrac{2}{5}\left(\tfrac{r_2}{r_1}\right)^5+\tfrac{2}{9}\left(\tfrac{r_2}{r_1}\right)^6\left(1+\tfrac{1}{\alpha r_1} \right) \\
G_{6\alpha} &=  -\tfrac{1}{3}\left( 1-\tfrac{2}{\alpha r_1}\right)+\left(\tfrac{r_2}{r_1}\right)^2-\tfrac{2}{3}\left(\tfrac{r_2}{r_1}\right)^3\left( 1+\tfrac{1}{\alpha r_1}\right)
\end{align*}

\end{lem}

\begin{rem}\label{convergenciaGi}
It is clear that $G_{i\alpha}$ converges to $G_i$, when \mbox{$\alpha\to\infty$} for $i=1,2,\dots,6$.
\end{rem}

\begin{thm}
i) For the distributed optimal control problems (\ref{PControlJ1}) and
(\ref{PControlalfaJ1}), the optimal solutions are given by:
\begin{equation}\label{goptimocorona3}
g_{op}=-\frac{G_{4}q\tfrac{r_2^2}{r_1}+G_5 (b-z_d)}{2r_1^2\left( G_1+ G_3 \tfrac{M_1}{r_1^4}\right)}
\end{equation}
and
\begin{equation}\label{goptimoalfacorona3}
g_{\alpha_{op}}=-\frac{G_{4\alpha}q\tfrac{r_2^2}{r_1}+G_{5\alpha} (b-z_d) }{2r_1^2\left( G_{1\alpha}+ G_{3\alpha} \tfrac{M_1}{r_1^4}\right)}.
\end{equation}
The optimal values corresponding to those optimal controls are given by the following formulas:
\begin{align}
J_1(g_{op})&= \left[ 4 \left( G_1+G_3\tfrac{M_1}{r_1^4}\right)\left(G_2 q^2 \tfrac{r_2^4}{r_1^2}+G_3(b-z_d)^2+G_6 q \tfrac{r_2^2}{r_1} (b-z_d) \right)\right. \nonumber\\
&\left.-\left( G_4 q\tfrac{r_2^2}{r_1}+G_5(b-z_d)\right)^2\right]\frac{\pi r_1^3}{2\left( G_1 +G_3\tfrac{M_1}{r_1^4}\right)} \label{J1gopCorona3}
\end{align}
and
\begin{align}
J_{1\alpha}(g_{\alpha_{op}})&= \left[ 4 \left( G_{1\alpha}+G_{3\alpha}\tfrac{M_1}{r_1^4}\right)\left(G_{2\alpha} q^2 \tfrac{r_2^4}{r_1^2}+G_{3\alpha}(b-z_d)^2+G_{6\alpha} q \tfrac{r_2^2}{r_1} (b-z_d) \right)\right. \nonumber\\
&\left.-\left( G_{4\alpha} q\tfrac{r_2^2}{r_1}+G_{5\alpha}(b-z_d)\right)^2\right]\frac{\pi r_1^3}{2\left( G_{1\alpha} +G_{3\alpha}\tfrac{M_1}{r_1^4}\right)} \label{J1alfagalfaopCorona3}
\end{align}

\noindent ii) For the boundary optimal control problems (\ref{PControlJ2}) and
(\ref{PControlalfaJ2}), the optimal solutions are given by:
\begin{equation}\label{qoptimocorona3}
q_{op}=-\frac{r_1}{2r_2^2}\frac{\left(G_4 g r_1^2+G_6(b-z_d)\right)}{\left( G_2+\tfrac{M_2}{r_1 r_2^2}\right)}
\end{equation}
and
\begin{equation}\label{qoptimoalfacorona3}
q_{\alpha_{op}}=-\frac{r_1}{2r_2^2}\frac{\left(G_{4\alpha} g r_1^2+G_{6\alpha}(b-z_d)\right)}{\left( G_{2\alpha}+\tfrac{M_2}{r_1 r_2^2}\right)}.
\end{equation}
The corresponding optimal values can be expressed by:
\begin{align}
J_2(q_{op})&= \left[4 \left( G_2+\tfrac{M_2}{r_1 r_2^2}\right) \left(G_1 g^2 r_1^4+ G_3(b-z_d)^2+G_5 g r_1^2 (b-z_d) \right)\right. \nonumber \\
&-\left.\left(G_4 g r_1^2+G_6 (b-z_d) \right)^2 \right]\frac{\pi r_1^3}{2\left( G_2+\tfrac{M_2}{r_1 r_2^2}\right)}\label{J2qopCorona3}
\end{align}
and
\begin{align}
J_{2\alpha}(q_{\alpha_{op}})&= \left[4 \left( G_{2\alpha}+\tfrac{M_2}{r_1 r_2^2}\right) \left(G_{1\alpha} g^2 r_1^4+ G_{3\alpha}(b-z_d)^2+G_{5\alpha} g r_1^2 (b-z_d) \right)\right.\nonumber \\
&-\left.\left(G_{4\alpha} g r_1^2+G_{6\alpha} (b-z_d) \right)^2 \right]\frac{\pi r_1^3}{2\left( G_{2\alpha}+\tfrac{M_2}{r_1 r_2^2}\right)}.\label{J2alfaqalfaopCorona3}
\end{align}

\noindent iii) For the boundary optimal control problems (\ref{PControlJ3}) and
(\ref{PControlalfaJ3}), the optimal controls are given by
\begin{equation}\label{boptimocorona3}
b_{op}=-\frac{ G_5 g r_1^2 +G_6 q \tfrac{r_2^2}{r_1} -2 G_3 z_d}{2\left( G_3+\tfrac{M_3}{r_1}\right)}
\end{equation}
and
\begin{equation}\label{boptimoalfacorona3}
b_{\alpha_{op}}=-\frac{ G_{5\alpha} g r_1^2 +G_{6\alpha} q \tfrac{r_2^2}{r_1}-2 G_{3\alpha} z_d }{2\left( G_{3\alpha}+\tfrac{M_3}{r_1}\right)}.
\end{equation}
Moreover,  $J_3(b_{op})$ and $J_{3\alpha}(b_{\alpha_{op}})$ can be obtained by the following formulas:
\begin{footnotesize}
\begin{align}
J_3(b_{op})&= \left[ 4 \left(G_3+\tfrac{M_3}{r_1}\right) \left( G_1 g^2 r_1^4 +G_2 q^2 \tfrac{r_2^4}{r_1^2}+G_3 z_d^2+G_4 g q r_1 r_2^2 -G_5 g r_1^2 z_d +G_6 q \tfrac{r_2^2}{r_1 }z_d\right)\right.\nonumber \\
&-\left. \left( -2 G_3 z_d +G_5 g r_1^2 +G_6 q \tfrac{r_2^2}{r_1}\right)^2 \right]\frac{\pi r_1^3}{2\left( G_3 +\tfrac{M_3}{r_1}\right)}\label{J3bopCorona3}
\end{align}
\end{footnotesize}
and
\begin{footnotesize}
\begin{align}
J_{3\alpha}(b_{\alpha_{op}})&= \left[ 4 \left(G_{3\alpha}+\tfrac{M_3}{r_1}\right) \left( G_{1\alpha} g^2 r_1^4 +G_{2\alpha} q^2 \tfrac{r_2^4}{r_1^2}+G_{3\alpha} z_d^2+G_{4\alpha} g q r_1 r_2^2 -G_{5\alpha} g r_1^2 z_d +G_{6\alpha} q \tfrac{r_2^2}{r_1 }z_d\right)\right.\nonumber \\
&-\left. \left( -2 G_{3\alpha} z_d +G_{5\alpha} g r_1^2 +G_{6\alpha} q \tfrac{r_2^2}{r_1}\right)^2 \right]\frac{\pi r_1^3}{2\left( G_{3\alpha} +\tfrac{M_3}{r_1}\right)}\label{J3alfabalfaopCorona3}
\end{align}
\end{footnotesize}

\noindent iv) For the distributed-boundary optimal control problem
(\ref{PControl}) and (\ref{PControlalfa}), the optimal solutions are given by
\begin{align}\label{bicontrolcorona3}
(g,q)_{op}=(g^{op},q^{op})=& \left( \frac{(b-z_d)}{r_1^2} \Delta_3, \frac{(b-z_d)r_1}{r_2^2} \Pi_3\right)
\end{align}
with
$${\Delta}_3=\tfrac{\left(G_4 G_6-2 G_5 \left( G_2 +\tfrac{M_5}{r_1 r_2^2}\right) \right)}{\left(4 \left( G_1 +G_3 \tfrac{M_4}{r_1^4}\right) \left( G_2 +\tfrac{M_5}{r_1 r_2^2}\right)-G_4^2 \right)}, \quad  {\Pi}_3=\tfrac{\left(G_4 G_5-2 G_6 \left( G_1+G_3 \tfrac{M_4}{r_1^4}\right) \right)}{\left(4\left(G_1+G_3 \tfrac{M_4}{r_1^4} \right)\left( G_2+\tfrac{M_5}{r_1 r_2^2}\right) -G_4^2\right)}$$
and
\begin{align}\label{bicontrolalfacorona3}
(g,q)_{\alpha_{op}}=(g_{\alpha}^{op},q_{\alpha}^{op})=& \left( \frac{(b-z_d)}{r_1^2}\Delta_{3_\alpha} ,\frac{(b-z_d)r_1}{r_2^2} \Pi_{3_\alpha}\right)
\end{align}
with
$$\Delta_{3_\alpha} =\tfrac{G_{4\alpha} G_{6\alpha}-2 G_{5\alpha} \left( G_{2\alpha} +\tfrac{M_5}{r_1 r_2^2}\right) }{4 \left( G_{1\alpha} +G_{3\alpha} \tfrac{M_4}{r_1^4}\right) \left( G_{2\alpha} +\tfrac{M_5}{r_1 r_2^2}\right)-G_{4\alpha}^2 },\quad \Pi_{3_\alpha}= \tfrac{G_{4\alpha} G_{5\alpha}-2 G_{6\alpha} \left( G_{1\alpha}+G_{3\alpha} \tfrac{M_4}{r_1^4}\right) }{4\left(G_{1\alpha}+G_{3\alpha} \tfrac{M_4}{r_1^4} \right)\left( G_{2\alpha}+\tfrac{M_5}{r_1 r_2^2}\right) -G_{4\alpha}^2}$$
Furthermore, $J_4$ at $(g,q)_{op}$ and $J_{4\alpha}$ at $(g,q)_{\alpha_{op}}$ can be computed by the following expressions:
\begin{align}
J_4(g^{op},q^{op})&= \left[ G_4 G_5 G_6+4\left( G_1+G_3 \tfrac{M_4}{r_1^4}\right)\left( G_2+\tfrac{M_5}{r_1 r_2^2}\right)G_3-\left(G_1+G_3 \tfrac{M_4}{r_1^4} \right) G_6^2\right. \nonumber \\
&-\left.\left( G_2+\tfrac{M_5}{r_1 r_2^2}\right)G_5^2-G_3 G_4^2\right] \frac{2\pi (b-z_d)^2 r_1^3}{\left(4 \left(G_1+G_3\tfrac{M_4}{r_1^4} \right)\left( G_2+\tfrac{M_5}{r_1 r_2^2}\right)-G_4^2 \right)}\label{J4gqopCorona3}
\end{align}
and
\begin{align}
J_{4\alpha}(g_{\alpha}^{op},q_{\alpha}^{op})&= \left[ G_{4\alpha} G_{5\alpha} G_{6\alpha}+4\left( G_{1\alpha}+G_{3\alpha} \tfrac{M_4}{r_1^4}\right)\left( G_{2\alpha}+\tfrac{M_5}{r_1 r_2^2}\right)G_{3\alpha}-\left(G_{1\alpha}+G_{3\alpha} \tfrac{M_4}{r_1^4} \right) G_{6\alpha}^2\right. \nonumber \\
&-\left.\left( G_{2\alpha}+\tfrac{M_5}{r_1 r_2^2}\right)G_{5\alpha}^2-G_{3\alpha} G_{4\alpha}^2\right] \frac{2\pi (b-z_d)^2 r_1^3}{\left(4 \left(G_{1\alpha}+G_{3\alpha}\tfrac{M_4}{r_1^4} \right)\left( G_{2\alpha}+\tfrac{M_5}{r_1 r_2^2}\right)-G_{4\alpha}^2 \right)}\label{J4alfagqalfaopCorona3}
\end{align}
\medskip
\noindent v) The estimates and convergences obtained in (v) of Theorem \ref{TeoremaRectangulo} are also verified for the spherical shell in $\mathbb{R}^3$. 
\end{thm}

\begin{proof}

i) Taking into account that the functional $J_1$ and $J_{1\alpha}$ can be expressed in the following quadratic forms:
\begin{align*}
J_1(g)&=2\pi \left[ (G_1 r_1^7+M_1 G_3 r_1^3)g^2 +\left(G_4  q r_1^4 r_2+G_5 r_1^5(b-z_d) \right)g\right.\\
&+\left.\left(G_2 q^2 r_1 r_2^4+G_3 r_1^3 (b-z_d)^2 +G_6 q r_1^2 r_2^2(b-z_d)\right) \right]
\end{align*}
and
\begin{align*}
J_{1\alpha}(g)&=2\pi \left[ (G_{1\alpha} r_1^7+M_1 G_{3\alpha} r_1^3)g^2 +\left(G_{4\alpha}  q r_1^4 r_2+G_{5\alpha} r_1^5(b-z_d) \right)g\right.\\
&+\left.\left(G_{2\alpha} q^2 r_1 r_2^4+G_{3\alpha} r_1^3 (b-z_d)^2 +G_{6\alpha} q r_1^2 r_2^2(b-z_d)\right) \right]
\end{align*}
it can be  obtained that the optimal solutions $g_{op}$ and $g_{\alpha_{op}}$ for the
problems (\ref{PControlJ1}) and (\ref{PControlalfaJ1}) are given by
(\ref{goptimocorona3}) and (\ref{goptimoalfacorona3}), respectively, since the second derivative is positive in both cases. The optimal values formulas (\ref{J1gopCorona3}) and (\ref{J1alfagalfaopCorona3}) are deduced by evaluating $J_1$ and $J_{1\alpha}$ at $g_{op}$ and $g_{\alpha_{op}}$, respectively.

\medskip
\noindent ii) The functional $J_2$ and $J_{2\alpha}$ are given by the expressions:
\begin{align*}
J_2(q)&=2 \pi \left[ \left(G_2 r_1 r_2^4+M_2 r_2^2\right)q^2+\left( G_4 r_1^4 r_2^2 g+G_6 r_1^2 r_2^2 (b-z_d)\right)q\right.\\
&\left.+\left( G_1 r_1^7 g^2+G_3 r_1^3 (b-z_d)^2+G_5 r_1^5 g (b-z_d)\right)\right]
\end{align*}
and
\begin{align*}
J_{2\alpha}(q)&=2 \pi \left[ \left(G_{2\alpha} r_1 r_2^4+M_2 r_2^2\right)q^2+\left( G_{4\alpha} r_1^4 r_2^2 g+G_{6\alpha} r_1^2 r_2^2 (b-z_d)\right)q\right.\\
&\left.+\left( G_{1\alpha} r_1^7 g^2+G_{3\alpha} r_1^3 (b-z_d)^2+G_{5\alpha} r_1^5 g (b-z_d)\right)\right].
\end{align*}
Therefore it is immediate that the optimal controls for problems (\ref{PControlJ2}) and (\ref{PControlalfaJ2}) are given by
(\ref{qoptimocorona3}) and (\ref{qoptimoalfacorona3}), respectively, since the second derivative is positive in both cases.

The computation of  $J_2(q_{op})$ and $J_{2\alpha}(q_{\alpha_{op}})$ leads to the closed formulas (\ref{J2qopCorona3}) and (\ref{J2alfaqalfaopCorona3}) for the optimal values of the control problems considered.

\medskip
\noindent iii)
For the problems (\ref{PControlJ3}) and (\ref{PControlalfaJ3}), the
functional $J_{3}$ and $J_{3\alpha}$ are given by
\begin{align*}
J_3(b)&= 2\pi \left[ \left(G_3 r_1^3+M_3 r_1^2\right)b^2+\left( -2G_3 r_1^3 z_d +G_5 r_1^5 g+G_6 r_1^2 r_2^2 q\right)b\right.\\
&\left.+ G_1 r_1^7 g^2+G_2 r_1 r_2^4 q^2+G_3 r_1^3 z_d^2 +G_4 r_1^4 r_2^2 g q -G_5 r_1^5 g z_d+G_6 r_1^2 r_2^2 q z_d\right]
\end{align*}
and
\begin{align*}
&J_{3\alpha}(b)= 2\pi \left[ \left(G_{3\alpha} r_1^3+M_3 r_1^2\right)b^2+\left( -2G_{3\alpha} r_1^3 z_d +G_{5\alpha} r_1^5 g+G_{6\alpha} r_1^2 r_2^2 q\right)b\right.\\
&\left.+ G_{1\alpha} r_1^7 g^2+G_{2\alpha} r_1 r_2^4 q^2+G_{3\alpha} r_1^3 z_d^2 +G_{4\alpha} r_1^4 r_2^2 g q -G_{5\alpha} r_1^5 g z_d+G_{6\alpha} r_1^2 r_2^2 q z_d\right].
\end{align*}
Therefore the optimal controls are given by
(\ref{boptimocorona3}) and (\ref{boptimoalfacorona3}), respectively, since the second derivative is positive in both cases.

The optimal values given by expressions (\ref{J3bopCorona3}) and (\ref{J3alfabalfaopCorona3}) are obtained by computing $J_3$ and $J_{3\alpha}$ at $b_{op}$ and $b_{\alpha_{op}}$ respectively.
\medskip

\noindent iv)
For the distributed-boundary optimal control problems
(\ref{PControl}) and (\ref{PControlalfa}), the functional $J_{4}$
can be expressed as
\begin{align*}
J_{4}(g,q)&=2 \pi\left[ (G_1r_1^7+M_4 G_3 r_1^3)g^2+(G_2r_1 r_2^4+M_5 r_2^2)q^2+G_4 r_1^4 r_2^2 gq\right.\\
&+G_5 r_1^5g(b-z_d)+\left. G_6 r_1^2 r_2^2 q (b-z_d)+G_3r_1^3(b-z_d)^2\right]
\end{align*}
and the functional $J_{4\alpha}$ is given by:
\begin{align*}
J_{4\alpha}(g,q)&=2 \pi\left[ (G_{1\alpha}r_1^7+M_4 G_{3\alpha} r_1^3)g^2+(G_{2\alpha}r_1 r_2^4+M_5 r_2^2)q^2+G_{4\alpha} r_1^4 r_2^2 gq\right.\\
&+G_{5\alpha} r_1^5g(b-z_d)+\left. G_{6\alpha} r_1^2 r_2^2 q (b-z_d)+G_{3\alpha}r_1^3(b-z_d)^2\right]
\end{align*}
from where it can be obtained that the optimal
solutions are given by (\ref{bicontrolcorona3}) and
(\ref{bicontrolalfacorona3}), respectively, due to the second partial derivative test.
Formulas (\ref{J4gqopCorona3}) and (\ref{J4alfagqalfaopCorona3}) are deduced by evaluating $J_4$ at $(g,q)_{op}$ and $J_{4\alpha}$ at $(g,q)_{\alpha_{op}}$.

\medskip
\noindent v)
The convergences and estimates of the optimal controls and the optimal values, when \mbox{$\alpha\to \infty$} are obtained  by taking into account the formulas given in (i)-(iv) and the Remark \ref{convergenciaGi}. The corresponding computations can be found in Appendix A. They are omitted here due to the fact that they become cumbersome.

\end{proof}

\section{Conclusions}

In this paper, two different steady-state heat conduction problems $S$ and $S_{\alpha}$, for the Poisson equation with constant internal energy $g$ and mixed boundary conditions have been considered. 
 The problem $S$ corresponds to the case when a constant temperature $b$ is prescribed in the portion $\Gamma_1$ of the boundary and a constant flux $q$ on $\Gamma_2$, while in the problem $S_\alpha$, a convective condition is imposed at $\Gamma_1$ with a heat transfer coefficient $\alpha$ and external temperature $b$.  
Different optimal control problems can be also considered:  a \emph{distributed} control
problem on the internal energy $g$, a \emph{boundary} optimal
control problem on the heat flux $q$, a  \emph{boundary} optimal
control problem on the external temperature $b$ and a
\emph{distributed-boundary} simultaneous optimal control problem on
the source $g$ and the flux $q$ have been defined. 
We have obtained explicitly the optimal values of these optimal control problems, already study theoretically in literature in a general framework, for the particular domains: a rectangle in $\mathbb{R}^{2}$, an annulus in $\mathbb{R}^{2}$ and a spherical shell in $\mathbb{R}^{3}$. We point out that this solutions provide a benchmark for testing the accuracy of numerical methods.
Also, the limit behaviour of the system state, adjoint state, optimal controls and optimal values for the optimal control problems defined from $S_{\alpha}$, when $\alpha \to\infty$ have been analysed;  concluding that they converge to the corresponding system state, adjoint state, optimal controls and optimal values for the optimal control problems defined from $S$. All these limits have been proved to present an order of convergence of $1/\alpha$ which can be considered as new results for these kind of elliptic optimal control problems. This estimate, obtained for this particular domains, make us to believe that it also holds for a more general domain, encouraging to prove it  analytically.
\section*{Acknowledgements}

The present work has been partially sponsored by the European Union's Horizon 2020 Research and Innovation Programme under the Marie Sklodowska-Curie grant agreement 823731 CONMECH, and by the Project PIP No. 0275 from CONICET-UA, Rosario, Argentina; by the Project ANPCyT PICTO Austral 2016 No. 0090 for the first and third authors; and by Project PPI No. 18/C468 from SECyT-UNRC, Río Cuarto, Argentina for the second author.

\newpage
\appendix
\numberwithin{equation}{section}
\section{Appendix }
\subsection*{Explicit solution for the domain $\Omega_2$ }


\medskip
\noindent \textbf{Order of convergence for $p_{\alpha}$}
\begin{align*}
\lim\limits_{\alpha \to \infty} & \alpha \vert\vert  p_{\alpha}-p\vert\vert_{L^2(\Omega_2)}=\left\lbrace\frac{\pi }{768 r_1^2} \Big[ \left(r_1^2-r_2^2\right) \Big(192 b^2 \left(r_1^2-r_2^2\right)^2\right. \\
&+96 b \left(r_1^2-r_2^2\right) \left(g \left(r_1^4-4 r_1^2 r_2^2+3 r_2^4\right)+3 q r_1^2 r_2-5 q r_2^3-4 r_1^2 z_d+4 r_2^2 z_d\right)\\
&+g^2 \left(13 r_1^4-80 r_1^2 r_2^2+115 r_2^4\right) \left(r_1^2-r_2^2\right)^2\\
& +4 g \left(r_1^2-r_2^2\right) \left(q \left(19 r_1^4 r_2-92 r_1^2 r_2^3+97 r_2^5\right)-24 z_d \left(r_1^4-4 r_1^2 r_2^2+3 r_2^4\right)\right)\\
&\left. +8 \left(q^2 \left(14 r_1^4 r_2^2-49 r_1^2 r_2^4+41 r_2^6\right)-12 q r_2 z_d \left(3 r_1^4-8 r_1^2 r_2^2+5 r_2^4\right)+24 z_d^2 \left(r_1^2-r_2^2\right)^2\right)\right)\\
&+24 r_2^3 \log \left(\frac{r_2}{r_1}\right) \Big(16 b \left(r_1^2-r_2^2\right) \left(g r_2 \left(r_2^2-r_1^2\right)+q \left(r_1^2-2 r_2^2\right)\right)\\
&+g^2 r_2 \left(12 r_2^2-5 r_1^2\right) \left(r_1^2-r_2^2\right)^2+4 g \left(r_1^2-r_2^2\right) \left(q \left(r_1^4-9 r_1^2 r_2^2+11 r_2^4\right)+4 r_2 z_d \left(r_1^2-r_2^2\right)\right)\\
& +4 q \left(q \left(3 r_1^4 r_2-12 r_1^2 r_2^3+10 r_2^5\right)-4 z_d \left(r_1^4-3 r_1^2 r_2^2+2 r_2^4\right)\right)\Big)\\
&\left.-48 r_2^6 \log ^2\left(\frac{r_2}{r_1}\right) \Big(g^2 \left(r_1^4-5 r_1^2 r_2^2+4 r_2^4\right)+4 g q r_2 \left(3 r_1^2-4 r_2^2\right)-4 q^2 \left(r_1^2-4 r_2^2\right)\Big)\Big]  \right\rbrace^{1/2}.
\end{align*}

\medskip

\noindent\textbf{Order of convergence for $g_{\alpha_{op}}$}

\begin{align*}
\lim\limits_{\alpha \to \infty} & \alpha \vert  g_{\alpha_{op}}-g_{op}\vert= \frac{12}{\mathcal{G}_2}\;\Bigl|4 r_2^3 \left(r_2^2-r_1^2\right) \log \left(\tfrac{r_2}{r_1}\right) \Big[ 36 r_2^3 (b-z_d) \left(r_1^2-r_2^2\right)-96 M_1 q \left(r_1^2-2 r_2^2\right)\\
& +q \left(4 r_1^6-13 r_1^4 r_2^2+14 r_1^2 r_2^4+7 r_2^6\right)\Big]\\
&-\left(r_1^2-r_2^2\right)^2 \Big[ 4 b \left(r_1^2-r_2^2\right) \left(-96 M_1+r_1^4-5 r_1^2 r_2^2+10 r_2^4\right) \\
&-96 M_1 \left(3 q r_1^2 r_2-5 q r_2^3-4 r_1^2 z_d+4 r_2^2 z_d\right)+3 q r_1^6 r_2-14 q r_1^4 r_2^3+19 q r_1^2 r_2^5+4 q r_2^7-4 r_1^6 z_d\\
&  +24 r_1^4 r_2^2 z_d-60 r_1^2 r_2^4 z_d+40 r_2^6 z_d\Big]-24 r_2^5 \log ^2\left(\tfrac{r_2}{r_1}\right) \Big[ 4 r_2 (b-z_d) \left(r_1^4-r_2^4\right)\\
& +q \left(2 r_1^6-7 r_1^4 r_2^2+6 r_1^2 r_2^4+r_2^6\right)\Big]+96 q r_1^2 r_2^9 \log ^3\left(\tfrac{r_2}{r_1}\right)\Bigr|
\end{align*}
with
\begin{align*}
\mathcal{G}_2&=r_1 \Big[ \left(r_1^2-r_2^2\right) \left(96 M_1+2 r_1^4-13 r_1^2 r_2^2+17 r_2^4\right)+\left(36 r_2^6-24 r_1^2 r_2^4\right) \log \left(\tfrac{r_2}{r_1}\right)-24 r_2^6 \log ^2\left(\tfrac{r_2}{r_1}\right)\Big]^2
\end{align*}

\medskip
\noindent\textbf{Order of convergence for $q_{\alpha_{op}}$}
\begin{align*}
\lim\limits_{\alpha \to \infty} & \alpha \vert  q_{\alpha_{op}}-q_{op}\vert=\frac{1}{\mathcal{Q}_2}\Bigl|(r_2^2 - r_1^2) \Big[  
    g (-24 M_2 r_1^2 + 3 r_1^4 r_2 + 40 M_2 r_2^2 - 4 r_1^2 r_2^3 + r_2^5) 
    -64 M_2 (b-z_d)\Big]  \\
& +  2 r_2^2 \Big[g (-3 r_1^4 r_2 + 2 r_1^2 r_2^3 + r_2^5 + 
       16 M_2 (r_1^2 - 2 r_2^2)) + 16 r_2 (r_1^2 - r_2^2) (b - z_d)\Big] \log \left(\tfrac{r_2}{r_1}\right) \\
& - 4 r_2^3 \Big[  
    g (-3 r_1^4 + 4 r_1^2 r_2^2 + r_2^4) - 8 (b-z_d) (r_1^2 + r_2^2) \Big]\log^2 \left(\tfrac{r_2}{r_1}\right) + 16 g r_1^2 r_2^5 \log^3 \left(\tfrac{r_2}{r_1}\right)
\Bigr|
\end{align*}
with
$$\mathcal{Q}_2=8 r_1 \Big[4 M_2-r_1^2 r_2+2 r_2^3 \log ^2\left(\tfrac{r_2}{r_1}\right)-2 r_2^3 \log \left(\tfrac{r_2}{r_1}\right)+r_2^3\Big]^2$$

\medskip
\noindent\textbf{Order of convergence for $b_{\alpha_{op}}$}
\begin{align*}
&\vert  b_{\alpha_{op}}-b_{op}\vert=\frac{1}{\alpha}\Big|\frac{\left(r_1^2-r_2^2\right) \left(g \left(r_1^2-r_2^2\right)+2 q r_2\right)}{2 r_1 \left(-2 M_3 r_1+r_1^2-r_2^2\right) }\Big|
\end{align*}

\medskip
\noindent\textbf{Order of convergence for $g_{\alpha}^{op}$ and $q_{\alpha}^{op}$}
\begin{align*}
\lim\limits_{\alpha \to \infty} & \alpha \vert  g_{\alpha}^{op}-g^{op}\vert=\frac{96 (b-z_d)}{\mathcal{P}_2} \; \Bigl|-4 r_2^3 \left(r_2^2-r_1^2\right)^2 \log \left(\tfrac{r_2}{r_1}\right) \Big[768 M_4 \Big(8 M_5 \left(r_1^2+6 r_2^2\right)-3 r_1^4 r_2+3 r_2^5\Big)\\
& +18432 M_5^2 r_2^3+64 M_5 \Big(8 r_1^6-49 r_1^4 r_2^2-16 r_1^2 r_2^4+21 r_2^6\Big)\\
& -3 \Big(3 r_1^8 r_2-65 r_1^6 r_2^3+49 r_1^4 r_2^5+17 r_1^2 r_2^7-4 r_2^9\Big)\Big]\\
& -\left(r_1^2-r_2^2\right)^3 \Big[-3072 M_4 \Big(64 M_5^2-14 M_5 r_1^2 r_2-6 M_5 r_2^3+r_1^4 r_2^2-r_2^6\Big) \\
& +2048 M_5^2 \Big(r_1^4-5 r_1^2 r_2^2+10 r_2^4\Big)-32 M_5 \Big(5 r_1^6 r_2-25 r_1^4 r_2^3+119 r_1^2 r_2^5-75 r_2^7\Big)\\
&  +r_2^2 \Big(5 r_1^8-10 r_1^6 r_2^2+156 r_1^4 r_2^4-182 r_1^2 r_2^6+31 r_2^8\Big)\Big]\\
& +16 r_2^3 \left(r_1^2-r_2^2\right) \log ^2\left(\tfrac{r_2}{r_1}\right) \Big[192 M_4 \left(32 M_5 \left(2 r_1^4-2 r_1^2 r_2^2-r_2^4\right)-r_2 \left(r_1^2-r_2^2\right)^2 \left(7 r_1^2+r_2^2\right)\right)\\
& -3072 M_5^2 r_2^3 \left(r_1^2+r_2^2\right)-32 M_5 \Big(4 r_1^8-12 r_1^6 r_2^2-53 r_1^4 r_2^4+3 r_1^2 r_2^6+4 r_2^8\Big)+5 r_1^{10} r_2-81 r_1^8 r_2^3\\
&  +25 r_1^6 r_2^5+46 r_1^4 r_2^7+6 r_1^2 r_2^9-r_2^{11}\Big]-256 r_1^4 r_2^6 \log ^4\left(\tfrac{r_2}{r_1}\right) \Big[-192 M_4 \left(r_1^2-r_2^2\right)+96 M_5 r_2^3+2 r_1^6+r_2^6\Big]\\
& -64 r_1^2 r_2^6 \log ^3\left(\tfrac{r_2}{r_1}\right) \Big[192 M_4 \left(r_1^4-r_2^4\right)-r_1^2 \left(288 M_5 r_2^3+7 r_2^6\right)-r_2^5 \left(96 M_5+r_2^3\right)+16 r_1^8-8 r_1^6 r_2^2\Big] \Bigr|
\end{align*}
and
\begin{align*}
\lim\limits_{\alpha \to \infty} & \alpha \vert  q_{\alpha}^{op}-q^{op}\vert=\frac{64 (b-z_d)}{\mathcal{P}_2} \; \Bigl| 2 \left(r_1^2-r_2^2\right)^3 \Big[147456 M_4^2 M_5\\
&-24 M_4 \Big(64 M_5 \left(5 r_1^4-13 r_1^2 r_2^2+2 r_2^4\right)-3 \left(9 r_1^6 r_2-31 r_1^4 r_2^3+7 r_1^2 r_2^5+15 r_2^7\right)\Big)\\
&  -\left(r_1^2-r_2^2\right) \Big(8 M_5 \left(r_1^6-15 r_1^4 r_2^2+3 r_1^2 r_2^4+83 r_2^6\right)+6 r_1^6 r_2^3-3 r_1^4 r_2^5-48 r_1^2 r_2^7+9 r_2^9\Big)\Big]\\
& +r_2^2 \left(r_1^2-r_2^2\right)^2 \log \left(\tfrac{r_2}{r_1}\right) \Big[147456 M_4^2 r_2 \left(r_1^2-r_2^2\right)\\
& -384 M_4 \left(r_1^2-r_2^2\right) \Big(48 M_5 \left(3 r_1^2-4 r_2^2\right)-4 r_1^4 r_2+35 r_1^2 r_2^3+11 r_2^5\Big)\\
&+r_2 \Big(4 r_1^6 \left(672 M_5 r_2+151 r_2^4\right)-2 r_1^4 \left(3744 M_5 r_2^3+209 r_2^6\right)+26 r_2^7 \left(96 M_5+r_2^3\right)\\
&-23 r_1^{10}-56 r_1^8 r_2^2-133 r_1^2 r_2^8\Big)\Big]\\
& -4 r_2^3 \left(r_1^2-r_2^2\right) \log ^2\left(\tfrac{r_2}{r_1}\right) \Big[-36864 M_4^2 \left(r_1^4-r_2^4\right)\\
& +192 M_4 \Big(6 r_1^4 \left(16 M_5 r_2+3 r_2^4\right)-r_1^2 \left(96 M_5 r_2^3+13 r_2^6\right)+96 M_5 r_2^5+10 r_1^8-14 r_1^6 r_2^2-r_2^8\Big)\\
&-r_1^8 \left(192 M_5 r_2+209 r_2^4\right)-9 r_1^6 \left(128 M_5 r_2^3-3 r_2^6\right)+r_1^4 \left(3840 M_5 r_2^5+137 r_2^8\right)\\
& -2 r_2^9 \left(96 M_5+r_2^3\right)+2 r_1^{12}+25 r_1^{10} r_2^2+20 r_1^2 r_2^{10}\Big]\\
&+192 r_1^4 r_2^7 \log ^4\left(\tfrac{r_2}{r_1}\right) \Big[192 M_4 \left(r_2^2-r_1^2\right)+96 M_5 r_2^3+2 r_1^6+r_2^6\Big]\\
&-48 r_1^2 r_2^5 \log ^3\left(\tfrac{r_2}{r_1}\right) \Big[192 M_4 \left(r_1^2-r_2^2\right)^2 \left(3 r_1^2+r_2^2\right)\\
& +r_2^2 \Big(-3 r_1^4 \left(96 M_5 r_2+r_2^4\Big)+8 r_1^2 \left(48 M_5 r_2^3+r_2^6\right)+96 M_5 r_2^5-14 r_1^8+8 r_1^6 r_2^2+r_2^8\right)\Big]\Bigr|
\end{align*}
with
\begin{align*}
\mathcal{P}_2&=r_1 \Big[-16 r_2^3 \log ^2\left(\tfrac{r_2}{r_1}\right) \Big(-192 M_4 \left(r_1^2-r_2^2\right)+96 M_5 r_2^3-4 r_1^6+r_2^6\Big)\nonumber\\
&  -8 r_2^3 \log \left(\tfrac{r_2}{r_1}\right) \Big(384 M_4 \left(r_1^2-r_2^2\right)+6 r_1^2 \left(32 M_5 r_2+r_2^4\right)-288 M_5 r_2^3-10 r_1^6+9 r_1^4 r_2^2-5 r_2^6\Big)\nonumber\\
&+\left(r_1^2-r_2^2\right) \Big(1536 M_4 \left(4 M_5-r_1^2 r_2+r_2^3\right)+64 M_5 \left(2 r_1^4-13 r_1^2 r_2^2+17 r_2^4\right)-5 r_1^6 r_2+51 r_1^4 r_2^3 \nonumber\\
&-75 r_1^2 r_2^5+29 r_2^7\Big)\Big]^2\label{constP2} 
\end{align*}

\medskip
\noindent\textbf{Order of convergence for $J_{1{\alpha}}(g_{\alpha_{op}})$}
\begin{align*}
&\lim\limits_{\alpha \to \infty} \alpha\;  \Bigl|  J_{1\alpha}(g_{\alpha_{op}})-J_1(g_{op})\Bigr|=\frac{\pi}{\mathcal{J}_{12}}  \Bigl| \Big[ 12 b r_1^6-60 b r_1^4 r_2^2+84 b r_1^2 r_2^4-48 b r_1^2 r_2^4 \log \left(\tfrac{r_2}{r_1}\right)  \\
&\quad   +48 b r_2^6 \log \left(\tfrac{r_2}{r_1}\right)-36 b r_2^6-192 M_1 q r_1^2 r_2+192 M_1 q r_2^3+5 q r_1^6 r_2-15 q r_1^4 r_2^3+24 q r_1^4 r_2^3 \log \left(\tfrac{r_2}{r_1}\right)\\
&\quad+3 q r_1^2 r_2^5+48 q r_1^2 r_2^5 \log ^2\left(\tfrac{r_2}{r_1}\right)-36 q r_1^2 r_2^5 \log \left(\tfrac{r_2}{r_1}\right)-12 q r_2^7 \log \left(\tfrac{r_2}{r_1}\right)+7 q r_2^7-12 r_1^6 z_d+60 r_1^4 r_2^2 z_d\\
&\quad  -84 r_1^2 r_2^4 z_d+48 r_1^2 r_2^4 z_d \log \left(\tfrac{r_2}{r_1}\right)-48 r_2^6 z_d \log \left(\tfrac{r_2}{r_1}\right)+36 r_2^6 z_d\Big] \Big[-768 M_1 q r_2^5 \log \left(\tfrac{r_2}{r_1}\right)-1536 b M_1 r_1^2 r_2^2 \\
&\quad + 768 b M_1 r_2^4+4 b r_1^8-40 b r_1^6 r_2^2+96 b r_1^4 r_2^4-96 b r_1^4 r_2^4 \log \left(\tfrac{r_2}{r_1}\right)-88 b r_1^2 r_2^6-192 b r_1^2 r_2^6 \log ^2\left(\tfrac{r_2}{r_1}\right)\\
&\quad +96 b r_1^2 r_2^6 \log \left(\tfrac{r_2}{r_1}\right)+28 b r_2^8+384 M_1 q r_1^4 r_2-768 M_1 q r_1^2 r_2^3+768 M_1 q r_1^2 r_2^3 \log \left(\tfrac{r_2}{r_1}\right)768 b M_1 r_1^4\\
&\quad  +384 M_1 q r_2^5-768 M_1 r_1^4 z_d+1536 M_1 r_1^2 r_2^2 z_d-768 M_1 r_2^4 z_d-q r_1^8 r_2+4 q r_1^6 r_2^3-8 q r_1^6 r_2^3 \log \left(\tfrac{r_2}{r_1}\right)\\
&\quad  -18 q r_1^4 r_2^5-48 q r_1^4 r_2^5 \log ^2\left(\tfrac{r_2}{r_1}\right)-24 q r_1^4 r_2^5 \log \left(\tfrac{r_2}{r_1}\right)+28 q r_1^2 r_2^7+24 q r_1^2 r_2^7 \log \left(\tfrac{r_2}{r_1}\right)+8 q r_2^9 \log \left(\tfrac{r_2}{r_1}\right)\\
&\quad   -13 q r_2^9-4 r_1^8 z_d+40 r_1^6 r_2^2 z_d-96 r_1^4 r_2^4 z_d+96 r_1^4 r_2^4 z_d \log \left(\tfrac{r_2}{r_1}\right)+88 r_1^2 r_2^6 z_d+192 r_1^2 r_2^6 z_d \log ^2\left(\tfrac{r_2}{r_1}\right)\\
&\quad -96 r_1^2 r_2^6 z_d \log \left(\tfrac{r_2}{r_1}\right)-28 r_2^8 z_d \Big] \Bigr|
\end{align*}
with
\begin{align*}
&\mathcal{J}_{12}=16 r_1 \left[\left(r_1^2-r_2^2\right) \left(96 M_1+2 r_1^4-13 r_1^2 r_2^2+17 r_2^4\right)+\left(36 r_2^6-24 r_1^2 r_2^4\right) \log \left(\tfrac{r_2}{r_1}\right)-24 r_2^6 \log ^2\left(\tfrac{r_2}{r_1}\right)\right]^2 
\end{align*}

\medskip
\noindent\textbf{Order of convergence for $J_{2{\alpha}}(q_{\alpha_{op}})$}
\begin{align*}
&\lim\limits_{\alpha \to \infty} \alpha\;  \Bigl|  J_{2\alpha}(q_{\alpha_{op}})-J_2(q_{op})\Bigr|=\frac{\pi}{\mathcal{J}_{22}}  \Bigl|\Big[   16 b r_1^2 r_2+32 b r_2^3 \log \left(\tfrac{r_2}{r_1}\right)-16 b r_2^3+32 g M_2 r_1^2-32 g M_2 r_2^2  \\
&\quad -5 g r_1^4 r_2+4 g r_1^2 r_2^3+16 g r_1^2 r_2^3 \log ^2\left(\tfrac{r_2}{r_1}\right)-8 g r_1^2 r_2^3 \log \left(\tfrac{r_2}{r_1}\right)-4 g r_2^5 \log \left(\tfrac{r_2}{r_1}\right)+g r_2^5-16 r_1^2 r_2 z_d\\
&\quad  -32 r_2^3 z_d \log \left(\tfrac{r_2}{r_1}\right)+16 r_2^3 z_d \Big] \Big[ 128 b M_2 r_1^2-128 b M_2 r_2^2-16 b r_1^4 r_2+32 b r_1^2 r_2^3+64 b r_1^2 r_2^3 \log ^2\left(\tfrac{r_2}{r_1}\right)\\
&\quad  -16 b r_2^5+16 g M_2 r_1^4-64 g M_2 r_1^2 r_2^2-64 g M_2 r_2^4 \log \left(\tfrac{r_2}{r_1}\right)+48 g M_2 r_2^4-g r_1^6 r_2+5 g r_1^4 r_2^3\\
&\quad  +8 g r_1^4 r_2^3 \log ^2\left(\tfrac{r_2}{r_1}\right)+6 g r_1^4 r_2^3 \log \left(\tfrac{r_2}{r_1}\right)-7 g r_1^2 r_2^5-4 g r_1^2 r_2^5 \log \left(\tfrac{r_2}{r_1}\right)-2 g r_2^7 \log \left(\tfrac{r_2}{r_1}\right)+3 g r_2^7\\
&\quad  -128 M_2 r_1^2 z_d+128 M_2 r_2^2 z_d+16 r_1^4 r_2 z_d-32 r_1^2 r_2^3 z_d-64 r_1^2 r_2^3 z_d \log ^2\left(\tfrac{r_2}{r_1}\right)+16 r_2^5 z_d \Big]\Bigr|
\end{align*}
with
\begin{align*}
&\mathcal{J}_{22}=512 r_1 \left[4 M_2-r_1^2 r_2+2 r_2^3 \log ^2\left(\tfrac{r_2}{r_1}\right)-2 r_2^3 \log \left(\tfrac{r_2}{r_1}\right)+r_2^3\right]^2
\end{align*}

\medskip
\noindent\textbf{Order of convergence for $J_{3{\alpha}}(b_{\alpha_{op}})$}
\begin{align*}
 \lim\limits_{\alpha \to \infty} &\alpha\;  \Bigl|  J_{3\alpha}(b_{\alpha_{op}})-J_3(b_{op})\Bigr|=\frac{ \pi\Bigl|  M_3 \left(g \left(r_1^2-r_2^2\right)+2 q r_2\right)\Bigr|}{\Bigl| 8   r_1 \left(-2 M_3 r_1+r_1^2-r_2^2\right) \Bigr| } \Bigl|   g r1^3 (r1^2 - r2^2) + 4 q r_1 r_2 (r_1^2 - r_2^2) \\
 &- 
 3 g r_1 r_2^2 (r_1^2 - r_2^2) - 8 r_1 (r_1^2 - r_2^2) z_d - 
 4 r_1 r_2^3 (-2 q + g r_2) \log\left(\tfrac{r_2}{r_1}\right)\Bigr|
\end{align*}

\medskip
\noindent\textbf{Order of convergence for $J_{4{\alpha}}( g_{\alpha}^{op}, q_{\alpha}^{op})$}
\begin{align*}
 &\lim\limits_{\alpha \to \infty} \alpha\; \Bigl|  J_{4\alpha}(g_{\alpha}^{op},q_{\alpha}^{op})-J_4(g^{op},q^{op})\Bigr|=\frac{4 \pi  (b - z_d)^2 }{\mathcal{J}_{42}}  \Bigl| \Big[ -384 M_4 r_1^4 r_2+768 M_4 r_1^2 r_2^3-768 M_4 r_1^2 r_2^3 \log \left(\tfrac{r_2}{r_1}\right) \\
 &\quad  +768 M_4 r_2^5 \log \left(\tfrac{r_2}{r_1}\right)-384 M_4 r_2^5+96 M_5 r_1^6-480 M_5 r_1^4 r_2^2+672 M_5 r_1^2 r_2^4-384 M_5 r_1^2 r_2^4 \log \left(\tfrac{r_2}{r_1}\right) \\
 &\quad  +384 M_5 r_2^6 \log \left(\tfrac{r_2}{r_1}\right)-288 M_5 r_2^6-5 r_1^8 r_2+32 r_1^6 r_2^3+48 r_1^6 r_2^3 \log ^2\left(\tfrac{r_2}{r_1}\right)+44 r_1^6 r_2^3 \log \left(\tfrac{r_2}{r_1}\right)\\
 &\quad -54 r_1^4 r_2^5-36 r_1^4 r_2^5 \log \left(\tfrac{r_2}{r_1}\right)+32 r_1^2 r_2^7-12 r_1^2 r_2^7 \log \left(\tfrac{r_2}{r_1}\right)+4 r_2^9 \log \left(\tfrac{r_2}{r_1}\right)-5 r_2^9\Big] \Big[6144 M_4 M_5 r_1^4 \\
 & \quad  -12288 M_4 M_5 r_1^2 r_2^2+6144 M_4 M_5 r_2^4-768 M_4 r_1^6 r_2+2304 M_4 r_1^4 r_2^3+3072 M_4 r_1^4 r_2^3 \log ^2\left(\tfrac{r_2}{r_1}\right)\\
 &\quad  -2304 M_4 r_1^2 r_2^5-3072 M_4 r_1^2 r_2^5 \log ^2\left(\tfrac{r_2}{r_1}\right)+768 M_4 r_2^7+32 M_5 r_1^8-320 M_5 r_1^6 r_2^2+768 M_5 r_1^4 r_2^4\\
 &\quad  -768 M_5 r_1^4 r_2^4 \log \left(\tfrac{r_2}{r_1}\right)-704 M_5 r_1^2 r_2^6-1536 M_5 r_1^2 r_2^6 \log ^2\left(\tfrac{r_2}{r_1}\right)+768 M_5 r_1^2 r_2^6 \log \left(\tfrac{r_2}{r_1}\right)+224 M_5 r_2^8\\
 &\quad  -r_1^{10} r_2+17 r_1^8 r_2^3+16 r_1^8 r_2^3 \log ^2\left(\tfrac{r_2}{r_1}\right)+24 r_1^8 r_2^3 \log \left(\tfrac{r_2}{r_1}\right)-46 r_1^6 r_2^5-24 r_1^6 r_2^5 \log \left(\tfrac{r_2}{r_1}\right)+46 r_1^4 r_2^7\\
 &\quad  -24 r_1^4 r_2^7 \log \left(\tfrac{r_2}{r_1}\right)-17 r_1^2 r_2^9-16 r_1^2 r_2^9 \log ^2\left(\tfrac{r_2}{r_1}\right)+24 r_1^2 r_2^9 \log \left(\tfrac{r_2}{r_1}\right)+r_2^{11} \Big]\Bigr|
 \end{align*}
 with 
 \begin{align*}
&\mathcal{J}_{42}= r_1 \left[6144 M_4 M_5 r_1^2-6144 M_4 M_5 r_2^2-1536 M_4 r_1^4 r_2+3072 M_4 r_1^2 r_2^3 +3072 M_4 r_1^2 r_2^3 \log ^2\left(\tfrac{r_2}{r_1}\right)\right.\\
&\quad -3072 M_4 r_1^2 r_2^3 \log \left(\tfrac{r_2}{r_1}\right)-3072 M_4 r_2^5 \log ^2\left(\tfrac{r_2}{r_1}\right)+3072 M_4 r_2^5 \log \left(\tfrac{r_2}{r_1}\right)-1536 M_4 r_2^5+128 M_5 r_1^6\\
&\quad -960 M_5 r_1^4 r_2^2+1920 M_5 r_1^2 r_2^4-1536 M_5 r_1^2 r_2^4 \log \left(\tfrac{r_2}{r_1}\right)-1536 M_5 r_2^6 \log ^2\left(\tfrac{r_2}{r_1}\right)+2304 M_5 r_2^6 \log \left(\tfrac{r_2}{r_1}\right)\\
&\quad  -1088 M_5 r_2^6-5 r_1^8 r_2+56 r_1^6 r_2^3+64 r_1^6 r_2^3 \log ^2\left(\tfrac{r_2}{r_1}\right)+80 r_1^6 r_2^3 \log \left(\tfrac{r_2}{r_1}\right)-126 r_1^4 r_2^5-72 r_1^4 r_2^5 \log \left(\tfrac{r_2}{r_1}\right)\\
&\quad \left. +104 r_1^2 r_2^7-48 r_1^2 r_2^7 \log \left(\tfrac{r_2}{r_1}\right)-16 r_2^9 \log ^2\left(\tfrac{r_2}{r_1}\right)+40 r_2^9 \log \left(\tfrac{r_2}{r_1}\right)-29 r_2^9\right]^2
 \end{align*}

\subsection*{Explicit solution for the domain $\Omega_3$ }

\medskip

\medskip
\noindent \textbf{Order of convergence for $p_{\alpha}$}
\begin{align*}
\lim\limits_{\alpha \to \infty} &\alpha \vert\vert  p_{\alpha}-p\vert\vert_{L^2(\Omega_3)}= \left\lbrace\frac{\pi (r_2-r_1)^3}{42525 r_1^6} \Big[6300 b^2 r_1^2 \Big(r_1^2+r_1r_2+r_2^2\Big)^3 \right.\\
&+420 b r_1 \Big(r_1^2+r_1 r_2+r_2^2\Big) \Big(4 g \left(r_1^5+4 r_1^4 r_2+10 r_1^3 r_2^2+14 r_1^2 r_2^3+11 r_1 r_2^4+5 r_2^5\right) (r_1-r_2)^2\\
&+3 \Big(q r_2^2 \left(7 r_1^4+14 r_1^3 r_2+6 r_1^2 r_2^2-7 r_1 r_2^3-20 r_2^4\right)-10 r_1 z_d \left(r_1^2+r_1 r_2+r_2^2\right)^2\Big)\Big)\\
&+4 g^2 \Big(31 r_1^8+217 r_1^7 r_2+868 r_1^6 r_2^2+2248 r_1^5 r_2^3+4018 r_1^4 r_2^4+5047 r_1^3 r_2^5+4336 r_1^2 r_2^6\\
&+2380 r_1 r_2^7+700 r_2^8\Big) (r_1-r_2)^4-24 g (r_1-r_2)^2 \Big(q r_2^2 \Big(-52 r_1^7-260 r_1^6 r_2-675 r_1^5 r_2^2\\
&-970 r_1^4 r_2^3-440 r_1^3 r_2^4+612 r_1^2 r_2^5+1085 r_1 r_2^6+700 r_2^7\Big)\\
&+70 r_1 z_d \left(r_1^2+r_1 r_2+r_2^2\right)^2 \left(r_1^3+3 r_1^2 r_2+6 r_1 r_2^2+5 r_2^3\right)\Big)\\
&+45 \Big(q^2 r_2^4 (r_1-r_2)^2 \left(71 r_1^4+355 r_1^3 r_2+771 r_1^2 r_2^2+952 r_1 r_2^3+560 r_2^4\right)\\
&-28 q r_1 r_2^2 z_d \left(7 r_1^6+21 r_1^5 r_2+27 r_1^4 r_2^2+13 r_1^3 r_2^3-21 r_1^2 r_2^4-27 r_1 r_2^5-20 r_2^6\right)\\
&\left. +140 r_1^2 z_d^2 \left(r_1^2+r_1 r_2+r_2^2\right)^3\Big)\Big]\right\rbrace^{1/2}
\end{align*}

\noindent\textbf{Order of convergence for $g_{\alpha_{op}}$}
\begin{align*}
\lim\limits_{\alpha \to \infty} &\alpha \vert  g_{\alpha_{op}}-g_{op}\vert= \frac{21 (r_2-r_1)}{\mathcal{G}_3}\;\Bigl| 4 (b-z_d) \left(r_1^2+r_1 r_2+r_2^2\right) \Big(1575 M_1 r_1^2 \left(r_1^2+r_1 r_2+r_2^2\right)^2 \\
&  -(r_1-r_2)^4 \left(4 r_1^6+24 r_1^5 r_2+84 r_1^4 r_2^2+199 r_1^3 r_2^3+354 r_1^2 r_2^4+420 r_1 r_2^5+175 r_2^6\right)\Big)\\
& -630 M_1 q r_2^2 r_1 \left(r_1^2+r_1 r_2+r_2^2\right)  \left(-7 r_1^4-14 r_1^3 r_2-6 r_1^2 r_2^2+7 r_1 r_2^3+20 r_2^4\right)\\
&-3 q r_2^2 (r_1-r_2)^4  \Big(7 r_1^7+49 r_1^6 r_2+146 r_1^5 r_2^2+198 r_1^4 r_2^3+105 r_1^3 r_2^4-170 r_1^2 r_2^5-265 r_1 r_2^6-70 r_2^7\Big) \Bigr|
\end{align*}
with
\begin{align*}
\mathcal{G}_3&=
4 \Big[315 M_1 r_1^2 \left(r_1^2+r_1 r_2+r_2^2\right)+\Big(2 r_1^4+10 r_1^3 r_2+30 r_1^2 r_2^2+49 r_1 r_2^3+35 r_2^4\Big) (r_1-r_2)^4\Big]^2
\end{align*}

\medskip
\noindent\textbf{Order of convergence for $q_{\alpha_{op}}$}
\begin{align*}
&\lim\limits_{\alpha \to \infty} \alpha \vert  q_{\alpha_{op}}-q_{op}\vert=\frac{r_2-r_1}{\mathcal{Q}_3}\Bigl|20 (b-z_d) \Big[6 M_2 r_1^2 \left(r_1^2+r_1 r_2+r_2^2\right)-r_2^2 (r_1-r_2)^3 \left(r_1^2-2 r_1 r_2-2 r_2^2\right)\Big]\\
&\quad +g (r_1-r_2)^2 \Big[4 M_2 r_1 \left(7 r_1^3+21 r_1^2 r_2+27 r_1 r_2^2+20 r_2^3\right)-r_2^2 (r_1-r_2)^3 \left(7 r_1^2+14 r_1 r_2+4 r_2^2\right)\Big]
\Bigr| \\
\end{align*}
with
$$\mathcal{Q}_3=40 \Big[r_2^2 (r_1-r_2)^3-3 M_2 r_1^2\Big]^2$$

\medskip
\noindent\textbf{Order of convergence for $b_{\alpha_{op}}$}
\begin{align*}
&\vert  b_{\alpha_{op}}-b_{op}\vert=\frac{1}{\alpha}\Bigl|\frac{\left(r_1^3-r_2^3\right) \left(g \left(r_1^3-r_2^3\right)+3 q r_2^2\right)}{3 r_1^2 \left(-3 M_3 r_1^2+r_1^3-r_2^3\right)}\Bigr|
\end{align*}

\medskip
\noindent\textbf{Order of convergence for $g_{\alpha}^{op}$ and $q_{\alpha}^{op}$}
\begin{align*}
&\lim\limits_{\alpha \to \infty} \alpha \vert  g_{\alpha}^{op}-g^{op}\vert=\frac{840 (r_2 - r_1) (b-z_d)}{\mathcal{P}_3} \; \Bigl|16800 M_4 \left(r_1^2+r_1 r_2+r_2^2\right) \Big[240 M_5^2 r_1^2 \left(r_1^2+r_1 r_2+r_2^2\right)^2 \\
& \quad  -2 M_5 r_2^2 \Big(59 r_1^4+55 r_1^3 r_2+27 r_1^2 r_2^2-56 r_1 r_2^3-40 r_2^4\Big) (r_1-r_2)^3+r_2^4 \left(15 r_1^2+2 r_1 r_2-2 r_2^2\right) (r_1-r_2)^6\Big]\\
&\quad -(r_1-r_2)^4 \Big[2560 M_5^2 \Big(4 r_1^8+28 r_1^7 r_2+112 r_1^6 r_2^2+307 r_1^5 r_2^3+637 r_1^4 r_2^4+973 r_1^3 r_2^5+949 r_1^2 r_2^6+595 r_1 r_2^7\\
&\quad  +175 r_2^8\Big)-8 M_5 r_2^2 \Big(417 r_1^6+1610 r_1^5 r_2+2915 r_1^4 r_2^2+1490 r_1^3 r_2^3-575 r_1^2 r_2^4+3448 r_1 r_2^5+1720 r_2^6\Big) (r_1-r_2)^3\\
&\quad  +r_2^4 \Big(297 r_1^4+525 r_1^3 r_2+433 r_1^2 r_2^2+256 r_1 r_2^3+64 r_2^4\Big) (r_1-r_2)^6\Big] \Bigr|
\end{align*}
and
\begin{align*}
&\lim\limits_{\alpha \to \infty} \alpha \vert  q_{\alpha}^{op}-q^{op}\vert=\frac{24 (r_2 - r_1) (b-z_d)}{\mathcal{P}_3} \; \Bigl|23520000 M_4^2 \left(r_1^2+r_1 r_2+r_2^2\right)^2 \Big[6 M_5 r_1^2 \left(r_1^2+r_1 r_2+r_2^2\right)\\
&\quad   -r_2^2 (r_1-r_2)^3 \left(r_1^2-2 r_1 r_2-2 r_2^2\right)\Big]-1400 M_4 \left(r_1^2+r_1 r_2+r_2^2\right) (r_1-r_2)^4 \Big[8 M_5 \Big(281 r_1^6+1686 r_1^5 r_2 \\
&\quad \left. +4431 r_1^4 r_2^2+6446 r_1^3 r_2^3+6441 r_1^2 r_2^4+4200 r_1 r_2^5+1400 r_2^6\right)-r_2^2 (r_1-r_2)^3 \left(489 r_1^4+2153 r_1^3 r_2 \right.\\
&\quad   +3105 r_1^2 r_2^2+600 r_1 r_2^3-152 r_2^4\Big)\Big]-(r_1-r_2)^8 \Big[320 M_5 \Big(4 r_1^8+44 r_1^7 r_2+264 r_1^6 r_2^2+1049 r_1^5 r_2^3 \\
& \quad  +2539 r_1^4 r_2^4+3495 r_1^3 r_2^5+2055 r_1^2 r_2^6-315 r_1 r_2^7-315 r_2^8\Big)-r_2^2 (r_1-r_2)^3 \Big(99 r_1^6+106 r_1^5 r_2 \\
&\quad    -1675 r_1^4 r_2^2-3270 r_1^3 r_2^3-3405 r_1^2 r_2^4-2304 r_1 r_2^5-576 r_2^6\Big)\Big] \Bigr|
\end{align*}
with
\begin{align*}
&\mathcal{P}_3=
\Big[33600 M_4 \left(r_1^2+r_1 r_2+r_2^2\right) \left(3 M_5 r_1^2-r_2^2 (r_1-r_2)^3\right)\\
&\quad  +(r_1-r_2)^4 \Big(320 M_5 \left(2 r_1^4+10 r_1^3 r_2+30 r_1^2 r_2^2+49 r_1 r_2^3+35 r_2^4\right)-r_2^2 (r_1-r_2)^3 \left(99 r_1^2+152 r_1 r_2+64 r_2^2\right)\Big)\Big]^2 
\end{align*}

\medskip
\noindent\textbf{Order of convergence for $J_{1{\alpha}}(g_{\alpha_{op}})$}
\begin{align*}
&\lim\limits_{\alpha \to \infty} \alpha\;  \Bigl|  J_{1\alpha}(g_{\alpha_{op}})-J_1(g_{op})\Bigr|=\frac{\pi   (r_2-r_1)}{\mathcal{J}_{13}}  \Bigl|  \Big[q\Big( 2520 M_1  r_1^3 r_2^2 + 33  r_1^7 r_2^2 - 2520 M_1  r_1^2 r_2^3 +  33  r_1^6 r_2^3 \\
&\quad - 2520 M_1  r_1 r_2^4 - 177  r_1^5 r_2^4 -  198  r_1^4 r_2^5 + 747  r_1^3 r_2^6 - 471  r_1^2 r_2^7 - 51  r_1 r_2^8 +  84  r_2^9\Big) \\
& \quad +  (b-z_d) \Big(56 r_1^8 + 56 r_1^7 r_2 + 56 r_1^6 r_2^2 - 280 r_1^5 r_2^3 -     280 r_1^4 r_2^4 + 224 r_1^3 r_2^5 + 224 r_1^2 r_2^6 + 224 r_1 r_2^7 -     280 r_2^8\Big) \Big]\\
&\quad  \Big[ q\Big( 2100 M_1  r_1^4 r_2^2 - 3  r_1^8 r_2^2 + 4200 M_1  r_1^3 r_2^3 -  6  r_1^7 r_2^3 + 21  r_1^6 r_2^4 - 2100 M_1  r_1 r_2^5 +  63  r_1^5 r_2^5\\
&\quad - 4200 M_1  r_2^6 - 210  r_1^4 r_2^6 +  168  r_1^3 r_2^7 + 21  r_1^2 r_2^8 - 81  r_1 r_2^9 + 27 q r_2^{10}\Big) \\
&\quad + ( b-z_d) \Big(4200 M_1 r_1^5 + 8 r_1^9 + 8400 M_1 r_1^4 r_2 + 16 r_1^8 r_2 +     12600 M_1 r_1^3 r_2^2 + 24 r_1^7 r_2^2 + 8400 M_1 r_1^2 r_2^3\\
&\quad -     88 r_1^6 r_2^3 + 4200 M_1 r_1 r_2^4 - 200 r_1^5 r_2^4 + 192 r_1^4 r_2^5 +     584 r_1^3 r_2^6 - 824 r_1^2 r_2^7 + 288 r_1 r_2^8\Big) \Big]
 \Bigr| 
\end{align*}
with
\begin{align*}
\mathcal{J}_{13}=80  \Big[315 M_1 r_1^2 \left(r_1^2+r_1 r_2+r_2^2\right)+\left(2 r_1^4+10 r_1^3 r_2+30 r_1^2 r_2^2+49 r_1 r_2^3+35 r_2^4\right) (r_1-r_2)^4\Big]^2
\end{align*}

\medskip
\noindent\textbf{Order of convergence for $J_{2{\alpha}}(q_{\alpha_{op}})$}
\begin{align*}
&\lim\limits_{\alpha \to \infty} \alpha\;  \Bigl|  J_{2\alpha}(q_{\alpha_{op}})-J_2(q_{op})\Bigr|=\frac{\pi}{\mathcal{J}_{23}} (r_2-r_1)^2 \; \Bigl|  \Big[ 
g \Big(-40 M_2 r_1^3 - 40 M_2 r_1^2 r_2 - 40 M_2 r_1 r_2^2\\
&\quad  + 11 r_1^4 r_2^2 -    29 r_1^3 r_2^3 + 21 r_1^2 r_2^4 + r_1 r_2^5 - 4 r_2^6\Big)+(b-z_d) \Big(-20 r_1^2 r_2^2 - 20 r_1 r_2^3 + 40 r_2^4\Big)\Big]\\
&\quad \Big[g \Big(-16 M_2 r_1^5 - 16 M_2 r_1^4 r_2 - 16 M_2 r_1^3 r_2^2 + 3 r_1^6 r_2^2 +  64 M_2 r_1^2 r_2^3 - 13 r_1^5 r_2^3 + 64 M_2 r_1 r_2^4 \\
&\quad + 20 r_1^4 r_2^4 - 80 M_2 r_2^5 - 10 r_1^3 r_2^5 - 5 r_1^2 r_2^6 + 7 r_1 r_2^7 - 2 r_2^8\Big)+(b-z_d) \Big(-240 M_2 r_1^3\\
&\quad - 240 M_2 r_1^2 r_2 - 240 M_2 r_1 r_2^2 + 60 r_1^4 r_2^2 -    180 r_1^3 r_2^3 + 180 r_1^2 r_2^4 - 60 r_1 r_2^5\Big) \Big]\Bigr| 
\end{align*}
with
\begin{align*}
\mathcal{J}_{23}=2400 \Big[r_2^2 (r_1-r_2)^3-3 M_2 r_1^2\Big]^2
\end{align*}

\medskip
\noindent\textbf{Order of convergence for $J_{3{\alpha}}(b_{\alpha_{op}})$}
\begin{align*}
&\lim\limits_{\alpha \to \infty} \alpha\;  \Bigl|  J_{3\alpha}(b_{\alpha_{op}})-J_3(b_{op})\Bigr|=\frac{2\pi}{45 r_1^2 \Bigl|-3 M_3 r_1^2+r_1^3-r_2^3\Bigr|} \; \Bigl| 2 g^2 M_3 r_1 \left(r_1^9-6 r_1^6 r_2^3+9 r_1^4 r_2^5-9 r_1 r_2^8+5 r_2^9\right)\\
&\quad -3 g M_3 r_1 \Big(10 r_1 z_d \left(r_1^3-r_2^3\right)^2-q r_2^2 (r_1-r_2)^3 \left(7 r_1^3+21 r_1^2 r_2+27 r_1 r_2^2+20 r_2^3\right)\Big)\\
&\quad +15 q r_2^2 \Big(3 M_3 r_1 \left(q r_2^2 (r_1-r_2)^2 (r_1+2 r_2)+2 r_1 z_d \left(r_1^3-r_2^3\right)\Big)-4 z_d \left(r_1^3-r_2^3\right)^2\right)     \Bigr|
\end{align*}

\medskip
\noindent\textbf{Order of convergence for $J_{4{\alpha}}( g_{\alpha}^{op}, q_{\alpha}^{op})$}
\begin{align*}
 &\lim\limits_{\alpha \to \infty} \alpha\; \Bigl|  J_{4\alpha}(g_{\alpha}^{op},q_{\alpha}^{op})-J_4(g^{op},q^{op})\Bigr|=\frac{16 \pi r_1  (b - z_d)^2 }{\mathcal{J}_{43}}  \Bigl| 35280000 M_4^2 r_2^2 (r_1+2 r_2) \left(r_1^2+r_1 r_2+r_2^2\right)^2 \\
 &\quad \Big(4 M_5 \left(r_1^2+r_1 r_2+r_2^2\right)-r_2^2 (r_1-r_2)^3\Big)\\
 &\quad -4200 M_4 \left(r_1^3-r_2^3\right) \left(4480 M_5^2 \left(r_1^2+r_1 r_2+r_2^2\right)^2 \left(r_1^3+3 r_1^2 r_2+6 r_1 r_2^2+5 r_2^3\right) \right.\\
 &\quad-8 M_5 r_2^2 \left(247 r_1^5+856 r_1^4 r_2+1963 r_1^3 r_2^2+2840 r_1^2 r_2^3+2567 r_1 r_2^4+1292 r_2^5\right) (r_1-r_2)^3 \\
 &\quad   +r_2^4 \left(207 r_1^3+301 r_1^2 r_2+177 r_1 r_2^2+50 r_2^3\right) (r_1-r_2)^6\Big)\\
 &\quad  +(r_1-r_2)^5 \Big(-\left(35840 M_5^2 \left(r_1^9+10 r_1^8 r_2+55 r_1^7 r_2^2+199 r_1^6 r_2^3+505 r_1^5 r_2^4+919 r_1^4 r_2^5+1195 r_1^3 r_2^6 \right. \right. \\
 &\quad  \left. +1060 r_1^2 r_2^7+601 r_1 r_2^8+180 r_2^9\right)-16 M_5 r_2^2 \left(711 r_1^7+4763 r_1^6 r_2+17621 r_1^5 r_2^2+40700 r_1^4 r_2^3 \right. \\
 &\quad  \left.+57025 r_1^3 r_2^4+44014 r_1^2 r_2^5+18712 r_1 r_2^6+3879 r_2^7\right) (r_1-r_2)^3\\
 &\quad  \left. +r_2^4 \left(891 r_1^5+2880 r_1^4 r_2+3755 r_1^3 r_2^2+2480 r_1^2 r_2^3+875 r_1 r_2^4+144 r_2^5\right) (r_1-r_2)^6\right)\Big)  \Bigr|
\end{align*}
with 
 \begin{align*}
&\mathcal{J}_{43}= \Big[33600 M_4 \left(r_1^2+r_1 r_2+r_2^2\right) \left(3 M_5 r_1^2-r_2^2 (r_1-r_2)^3\right)\\
& +(r_1-r_2)^4 \Big(320 M_5 \left(2 r_1^4+10 r_1^3 r_2+30 r_1^2 r_2^2+49 r_1 r_2^3+35 r_2^4\right)-r_2^2 (r_1-r_2)^3 \left(99 r_1^2+152 r_1 r_2+64 r_2^2\right)\Big)\Big]^2
\end{align*}


\end{document}